\documentclass[12pt]{article}
\usepackage{geometry}
\usepackage{amsmath,amsfonts,amsthm,amssymb,amscd}

\usepackage{times}

\usepackage[utf8]{inputenc}
\usepackage{algorithm}
\usepackage{array}
\usepackage{algorithmicx}
\usepackage{algpseudocode}
\usepackage{graphicx}
\usepackage{letltxmacro}

\usepackage{framed}

\usepackage{color}
\usepackage[table,xcdraw,svgnames,dvipsnames]{xcolor}
\definecolor{light-salmon}{RGB}{255,140,120}
\usepackage[colorlinks]{hyperref}
\hypersetup{
	allbordercolors=.,
	citecolor=DarkGreen,
	linkcolor=DarkRed,
	urlcolor=NavyBlue
}

\graphicspath{{../DiamConstr/pics/}{./pics/}}

\usepackage{enumitem}

\numberwithin{equation}{section}
\theoremstyle{plain}
\newtheorem{thm}{Theorem}
\newtheorem{lemma}{Lemma}

\newtheorem{prop}[thm]{Proposition}
\newtheorem{deff}[thm]{Definition}
\theoremstyle{definition}
\newtheorem{rem}[thm]{Remark}
\newtheorem{conj}[thm]{Conjecture}

\newcommand{\Area}{\operatorname{Area}}

\newcommand{\diam}{\operatorname{diam}}
\renewcommand{\Bbb}{\mathbb}
\newcommand{\bo}[1]{{\bf #1}}

\newcommand{\Ver}{\operatorname{vert}}

\makeatletter
\DeclareFontFamily{U}{tipa}{}
\DeclareFontShape{U}{tipa}{m}{n}{<->tipa10}{}
\newcommand{\arc@char}{{\usefont{U}{tipa}{m}{n}\symbol{62}}}%

\newcommand{\arc}[1]{\mathpalette\arc@arc{#1}}

\newcommand{\arc@arc}[2]{%
	\sbox0{$\m@th#1#2$}%
	\vbox{
		\hbox{\resizebox{\wd0}{\height}{\arc@char}}
		\nointerlineskip
		\box0
	}%
}
\makeatother

\title{Isoperimetric problems related to extremal diameter graphs in 3D: theoretical and numerical aspects}

\author{Beniamin Bogosel\thanks{Faculty of Exact Sciences, Aurel Vlaicu
University of Arad, 2 Elena Dr\u agoi Street, Arad, Romania.\newline   
Email: \texttt{beniamin.bogosel@uav.ro}. Website: \url{https://beniamin-bogosel.github.io/}}}

\date{}

\begin{document}
	
\maketitle

\begin{abstract}
We study optimization problems for separable functionals of the Euclidean or
spherical lengths of dual edge pairs in finite extremal unit-diameter
configurations in three dimensions. For a fixed diameter graph, these
problems lead to nonconvex constrained optimization of the vertex
coordinates. We first prove that a convergent sequence of extremal
configurations retains an extremal geometric core after coincident points are
merged and vertices incident to at most one diameter are removed. A spherical
Crofton argument then gives a sharp lower bound for additive concave
functionals, attained by the regular tetrahedron. We also analyze the effect
of inserting or deleting dangling vertices. For the sum of products of
spherical dual-edge lengths, we obtain an exact supremal reformulation of the
three-dimensional Blaschke--Lebesgue area problem. The numerical study uses
all 10,644 available extremal configurations with at most 16 vertices. We
combine direct evaluation with gradient-based local optimization on each
fixed graph and reconstruct the intrinsic diameter graph and its dual pairs
after vertex collisions. For every supplied graph, the selected verified
endpoint contains a regular tetrahedron. These computations do not certify
the global maximum for any fixed graph, but they motivate structural
conjectures connecting tetrahedral containment with the
Blaschke--Lebesgue problem.
\end{abstract}



\section{Introduction}

Consider a finite set of $m \geq 4$ distinct points $X_m = \{x_1,x_2,\ldots,x_m\}$ in $\Bbb{R}^3$. We normalize its diameter to one. Since the points determine $m(m-1)/2$ segments, a natural question is: \emph{what is the largest possible number of diametral segments?} Equivalently, what is the maximum cardinality of
\begin{equation}\label{eq:diam-pairs} \mathcal D = \{(i,j) : 1\leq i<j \leq m,\ |x_i-x_j|=1\}?
	\end{equation}

This problem, called the \emph{V\'azsonyi problem}, was solved independently by Gr\"unbaum, Heppes, and Straszewicz; see \cite{disk-polygons} for the extension by Kupitz, Martini, and Perles and \cite{meissner_hynd} for a recent exposition. The answer is

\emph{The finite set $X_m$ containing $m$ points can have at most $2m-2$ diametral pairs.} 

An $m$-point set $X_m\subset\Bbb R^3$ with unit diameter and $2m-2$
diametral pairs is called an \bo{extremal set of unit diameter}, or simply an
\emph{extremal set}.  We denote by $B(X)$ the intersection of the unit balls
centered at the points of $X$. Extremal sets give rise to \emph{ball polyhedra}
and \emph{Meissner polyhedra}, defined below together with the properties used
in this work.

	\begin{framed}
	  \bo{Ball polyhedra}: if $X_m$ is extremal, then $B(X_m)$ is a polyhedral-like shape whose faces are regions of unit spheres and whose edges are circular arcs.  Their structure is described in \cite{disk-polygons,meissner_hynd}.  By analogy with two-dimensional Reuleaux polygons, they are also called Reuleaux polyhedra or Reuleaux polytopes \cite{bodies_of_constant_width}.  These shapes have the following properties:
	\begin{itemize}[noitemsep]
		\item The points $x_i$ in the extremal set $X_m$ are called \bo{essential vertices}. In \cite{meissner_hynd} a distinction is made between \bo{principal vertices} (endpoints for at least three diameters) or \bo{dangling vertices} (endpoints for exactly two diameters). 
		\item The \bo{face} opposite to $x_i$ in $B(X_m)$, denoted $F_i$, verifies $F_i = B(X_m) \cap \partial B(x_i)$ and is a region of a sphere, bounded by circular arcs. The face $F_i$ opposite to $x_i$ has angular points determined by $\{ x_j \in X_m : |x_i-x_j| =1\}$. Compared to classical polyhedra, it is possible to have a face of a ball polyhedron having only two vertices! See \cite{meissner_hynd} or the following sections of the article for examples. 
		Faces have the following involutive property: $x_j \in F_i \Longleftrightarrow x_i \in F_j$.
		\item Two adjacent faces $F_i, F_j$ in $B(X_m)$ meet along an \bo{edge} $(x_k,x_l)$. Since $F_i\cap F_j$ is an intersection of two unit spheres with distinct vertices, the corresponding edge is an arc $\arc{x_kx_l}$ of a circle with radius in $(0,1)$. 
		
		The relation $F_i\cap F_j = \arc{x_kx_l}$ implies $F_k\cap F_l = \arc{x_ix_j}$. The two exposed edges obtained in this way are called \bo{dual edges}; every exposed edge has a unique exposed dual edge \cite[Theorem~8.1]{disk-polygons}.  Their four cross distances are diameters:
		\[ |x_i-x_k|=|x_i-x_l|=|x_j-x_k|=|x_j-x_l|=1.\]
		The converse requires the cyclic-order and exposure conditions on the
		sphere-intersection circle.  Four cross distances alone only identify
		prospective common diameter-neighbours and do not guarantee that either
		joining arc is exposed.
		
		\item Ball polyhedra satisfy Euler's formula $F+V=E+2$. For $B(X_m)$, one has $F=V=m$ and $E=2m-2$.
		
		\item The edges of a ball polyhedron $B(X_m)$ can be grouped into $m-1$ pairs of dual edges denoted by
		\begin{equation}\label{eq:dual-edge-notation} (e_1,e_1'),\ldots,(e_{m-1},e_{m-1}').
		\end{equation}
		Given an edge $e_i$ with endpoints $x,y$ its \emph{spherical length} is denoted $\theta(e_i) = 2\arcsin(|x-y|/2)$. It is the length of a geodesic linking $x$ and $y$ on a sphere of unit radius. Note that $|x-y|\leq 1$ so $\theta(e_i) \in [0,\pi/3]$.
		\item For a pair of dual edges $(e,e')$ we can also consider their effective lengths in the ball polyhedron. Recall the notation in \cite{bogosel_Meissner}: let $\phi(e), \phi(e')$ be the dihedral angles associated with $e,e'$ in the tetrahedron generated by their vertices. These dihedral angles satisfy
		\[ \phi(e) = 2\arcsin \left( \frac{\sin \frac{\theta(e')}{2}}{\cos \frac{\theta(e)}{2}}\right), \phi(e') = 2\arcsin \left( \frac{\sin \frac{\theta(e)}{2}}{\cos \frac{\theta(e')}{2}}\right).\]
		The lengths of $e$ and $e'$ are denoted by $\ell(e), \ell(e')$, respectively and are given by
		\begin{equation}\label{eq:length-edge} \ell(e) = \phi(e')\cos \frac{\theta(e')}{2}, \ell(e') = \phi(e)\cos \frac{\theta(e)}{2}.
		\end{equation}
		\item The Euclidean edge length of an edge $e$ with endpoints $x_i$, $x_j$ will be denoted by $L(e) = |x_i-x_j|$.
	\end{itemize}
	\end{framed}

We recall that a convex shape $K \in \Bbb{R}^3$ has \bo{constant width} if the distance between any parallel supporting planes to $K$ is constant. The definition extends to all dimensions $d \geq 2$. It is obvious that balls have constant width, but there are many other shapes of constant width. It turns out that extremal finite sets of unit diameter play an essential role in constructing simple, yet rich classes of bodies of constant width in $\Bbb{R}^3$.

Sallee observed in \cite{sallee_polytopes} that ball polyhedra, or Reuleaux polytopes, are dense in the class of bodies of unit constant width. However, no ball polytope itself has constant width; see \cite{meissner_hynd}.

\begin{framed}
	\bo{Meissner polyhedra.} Any ball polyhedron $B(X_m)$ can be used to generate a shape of constant width as follows. Given $X_m$, consider the $m-1$ pairs of dual edges labeled like in \eqref{eq:dual-edge-notation}. Define
	\begin{equation}\label{eq:meiss-poly}
		 M = B(X_m \cup e_1\cup\ldots \cup e_{m-1}).
	\end{equation}
	Then $M$ has constant width. These polyhedra were introduced in \cite{montejano} and detailed descriptions and proofs regarding their properties are given in \cite{meissner_hynd}.
	\begin{itemize}[noitemsep,topsep=0pt]
		\item From each pair $(e_i,e_i')_{i=1,\ldots,m-1}$, either edge can be chosen in \eqref{eq:meiss-poly}. Thus an extremal finite set of unit diameter gives $2^{m-1}$ Meissner-polyhedron constructions, not necessarily all distinct. 
		\item The simplest Meissner polyhedron is the Meissner tetrahedron introduced by Meissner in \cite{Meissner-Schilling}. It is based on the regular tetrahedron. Among the $2^3$ choices among pairs of dual edges in \eqref{eq:meiss-poly} we find only two distinct Meissner tetrahedra. 
		\item The intersection of unit spheres centered on the edge $e_i$ of the ball polyhedron $B(X_m)$ creates a particular rotation body around the dual edge $e_i'$. This surface is called a \emph{spindle} in \cite{meissner_hynd}.
		\item The area and volume of Meissner polyhedra defined by \eqref{eq:meiss-poly} were computed in \cite{hynd-vol-per,bogosel_Meissner}. The second reference gives the simplified formula
		\begin{equation}\label{eq:area-Meiss-poly}
			|\partial M| = 2\pi - 2\sum_{i=1}^{m-1}f(\theta(e_i),\theta(e_i')),
		\end{equation}
		where $\theta(e_i), \theta(e_i')$ are spherical lengths of the pairs of dual edges in \eqref{eq:dual-edge-notation}. The function $f$ is given by
		\[ f:[0,\pi/3]^2 \to \Bbb{R}: f(x,y) = y\cos \frac{y}{2}\arcsin \left( \frac{\sin \frac{x}{2}}{\cos \frac{y}{2}}\right).\]
		Observe that this formula is completely explicit and separable in terms of the pairs of dual edges.
		\item Moreover, the expression in \eqref{eq:area-Meiss-poly} corresponding to a pair of dual edges has a geometric meaning:
		\begin{equation}
			 2f(\theta(e_i),\theta(e_i')) = \theta(e_i') \ell(e_i),
			 \label{eq:formula-prod-sph-ball}
		\end{equation}
		where $\ell(e_i)$ is the length of the ball polyhedron edge defined in \eqref{eq:length-edge}.
		\item By the above construction it follows that any ball polyhedron based on an extremal set contains a Meissner polyhedron, which is a shape of constant width. Therefore if $X$ is extremal, $B(X)$ has width at least equal to one in every direction.
	\end{itemize}
	\end{framed}

The study of lower bounds for the volume of bodies of constant width in
dimension $d\geq3$ gives rise to many open questions.  In dimension three,
the classical lower bound of Chakerian \cite{Chakerian} was recently improved
by Nishioka using spectral estimates and a dual formulation
\cite{Nishioka2026}.  In the opposite direction, recent constructions give
small-volume bodies of constant width in high dimensions
\cite{ArmanEtAlSmallVolume} and bodies with the symmetries of a regular
simplex, including a new tetrahedrally symmetric three-dimensional example
\cite{ArmanEtAlTetrahedral}.  A remarkable conjecture regarding
three-dimensional bodies of constant width is stated below.

\begin{conj}\label{conj:meissner}
  The Meissner tetrahedron minimizes the volume, or equivalently the area, among all three-dimensional bodies of fixed constant width.
\end{conj}

This conjecture was stated by Bonnesen and Fenchel \cite{Bonnesen-Fenchel}. The discussion in \cite{kawohl-webe} and simulations in \cite{AntunesBogosel22} further reinforce it. Moreover, \cite{bogosel_Meissner} proves the optimality of the tetrahedron among Meissner pyramids. Proving that every Meissner polyhedron has surface area at least that of the Meissner tetrahedron would constitute progress toward the conjecture. The objective of this paper is to reinforce the conjecture and provide partial results and equivalent formulations.

For a given extremal set $X_m$ having $m \geq 4$ points we denote like in \eqref{eq:diam-pairs} by $\mathcal D$ the pairs of indices corresponding to pairs of points realizing the diameters. Furthermore, denote by 
\begin{equation}\label{eq:edge-pairs}
	\mathcal E = \{ (i,j,k,l)  : 1\leq i,j,k,l \leq m, (x_i,x_j) \text{ and }(x_k,x_l)\text{ are dual edges}\}.
\end{equation}
In the set $\mathcal{E}$ we only consider unique dual pairs: $(e_i,e_i')$ and $(e_i',e_i)$ are considered the same dual pair. 

For a scalar summand $F$, define the dual-edge functional
\begin{equation}\label{eq:general-functional}
 J_F(X_m)=\sum_{i=1}^{m-1}F\bigl(L(e_i),L(e_i')\bigr).
\end{equation}
We study both the infimum and the supremum of \eqref{eq:general-functional}
over extremal sets.  This includes the area functional
\eqref{eq:area-Meiss-poly}, total edge lengths, and the product-type
functionals considered below.  Since a dual pair is intrinsically unordered,
$F$ is understood to be symmetric unless an orientation rule is explicitly
specified; an asymmetric summand is not well defined on an unordered pair by
itself.

For a fixed diameter graph $\mathcal D$ and a fixed dual-pair list
$\mathcal E$, the basic minimization problem considered in this paper has the
form
\begin{equation}\label{eq:intro-fixed-graph-minimization}
\begin{aligned}
 \underset{x_1,\ldots,x_m\in\mathbb R^3}{\operatorname{minimize}}\quad
 &J_F(X_m),\\
 \text{subject to}\quad
 &|x_i-x_j|^2=1 &&(ij\in\mathcal D),\\
 &|x_i-x_j|^2\leq1 &&(1\leq i<j\leq m),\\
 &\sum_{i=1}^m x_i=0.
\end{aligned}
\end{equation}
The maximization problems have the same feasible set.  The centroid condition
only removes translations; the quadratic diameter constraints and, in
general, the objective remain nonconvex.

Problem \eqref{eq:intro-fixed-graph-minimization} is related to Euclidean
distance geometry, where one seeks a realization of an edge-weighted graph in
a prescribed Euclidean dimension
\cite{LibertiDistanceGeometry,MoreWuGlobalContinuation}.  Here the equations
$|x_i-x_j|=1$ realize the diameter graph in $\mathbb R^3$, so finding a
feasible starting configuration is itself a distance-graph embedding problem.
Distance-geometry formulations are typically nonconvex and are approached by
local or global nonlinear optimization, often with several initializations or
with continuation, relaxation, or graph-based formulations
\cite{MoreWuProtein,LibertiCycleFormulations}.  The present problem has
additional structure: every unspecified distance is bounded above by one,
$\mathcal D$ is required to be an extremal diameter graph, the exposed
dual-edge incidences must be realized geometrically, and the objective
optimizes distances associated with these dual pairs rather than only a
distance-realization error.  General weighted-graph embeddability is strongly
NP-hard \cite{SaxeGraphEmbedding}; this fact motivates the optimization point
of view but does not, by itself, give a complexity result for the special
extremal graphs considered here.

\bo{Contributions.} The main contributions of the paper are the following.
\begin{itemize}[topsep=2pt,noitemsep]
	\item We analyze convergent sequences of extremal sets and identify the
\emph{extremal core} of the limit.  This provides existence results for
classes of optimization problems depending on finite extremal sets.
	\item We prove a sharp isoperimetric inequality for the total spherical
length of the dual edges, with equality for the regular tetrahedron, together
with extensions to concave functions of the spherical lengths.
	\item We give an exact reformulation of the three-dimensional
Blaschke--Lebesgue problem as the maximization, over finite extremal sets, of
the total spherical product
	\begin{equation}
		 \sum_{i=1}^{m-1}\theta(e_i)\theta(e_i'),
		 \label{eq:prod-spherical}
	\end{equation}
where $(e_i,e_i')$ are the pairs of dual edges.
	\item We perform direct evaluations and constrained gradient-based
optimization on the 10,644 supplied extremal diameter graphs with at most 16
vertices.  The results suggest the conjectures summarized in
Table~\ref{tab:direct-objectives}.
	\item We investigate in detail the maximization of
\eqref{eq:prod-spherical} for a fixed diameter graph.  The simulations indicate
containment of a regular tetrahedron in the selected maximizing candidates, a
property which suggests new strategies for the three-dimensional
Blaschke--Lebesgue problem.
\end{itemize}

\bo{Structure of the paper.} Section~2 proves a compactness result for
geometric limits, gives a compact formulation for a fixed labelled diameter
graph, and recalls the linked spherical partition.
Section~3 proves a sharp additive inequality, analyzes dangling subdivisions,
and relates three product-like functionals to mixed surface area and spherical
geometry.  Section~4 describes
the numerical method, the treatment of collapsed vertices, and the evidence
for tetrahedral decorations.

Various conjectures and numerical simulations will be presented using the
database of Meissner polyhedra in \cite{meissner_graphs} and the additional
data source for $m=15,16$ described in Section~4.1.

\section{Geometric limits and linked partitions}

\subsection{Geometric limits and the extremal core}

For a finite unit-diameter set $Z$, write $D(Z)$ for its diameter graph.  Its
\emph{diameter core}, denoted $\operatorname{core}(Z)$, is the set obtained by repeatedly
deleting vertices of degree zero or one from $D(Z)$.  Equivalently,
$D(\operatorname{core}(Z))$ is the graph-theoretic $2$-core of $D(Z)$.  Repetition is
important: when a leaf is deleted, a neighbouring vertex of degree two
becomes a new leaf and must also be removed.

We use the following local stability property of ball polyhedra.  Recall
that a principal vertex belongs to at least three distinct facets, whereas a
dangling vertex is a centre belonging to exactly two facets.

\begin{lemma}[persistence of principal vertices]
\label{lem:principal-stability}
Let $X^n=(x_1^n,\ldots,x_m^n)$ be labelled extremal unit-diameter sets and
suppose that $x_i^n\to x_i$ for every $i$.  Let $X$ be the set of distinct
points among $x_1,\ldots,x_m$.  If $p$ is a principal vertex of $B(X)$, then
there are principal vertices $p_n$ of $B(X^n)$ such that $p_n\to p$.  In
particular, $p\in X$.
\end{lemma}

\begin{proof}
Put $K_n=B(X^n)$ and $K=B(X)$.  These convex bodies converge in the
Hausdorff metric.  Indeed, if
\[
 f_n(z)=\max_i|z-x_i^n|,\qquad f(z)=\max_i|z-x_i|,
\]
then
$\lVert f_n-f\rVert_\infty\leq\max_i|x_i^n-x_i|\to0$.  Jung's theorem gives
a point $c$ and an $r<1$ such that $|c-x_i|\leq r$ for every $i$.  For
$z\in K$ and $0<t<1$, convexity of $f$ gives
\[
 f((1-t)z+tc)\leq(1-t)f(z)+tf(c)\leq1-t(1-r).
\]
Put $\epsilon_n=\lVert f_n-f\rVert_\infty$ and, for all large $n$, take
$t_n=2\epsilon_n/(1-r)<1$.  Then
$f_n((1-t_n)z+t_nc)\leq1-\epsilon_n<1$, uniformly for $z\in K$, and
the displacement tends uniformly to zero.  This proves
$K\subset K_n+o(1)B$.  Conversely, the sets $K_n$ are uniformly bounded,
and every limit of points $z_n\in K_n$ belongs to $K$ by uniform
convergence of $f_n$.  Hence $K_n\to K$ in the Hausdorff metric.

For a convex body $C$ and $q\in C$, write
\[
 N_C(q)=\{u:\langle u,z-q\rangle\leq0\text{ for every }z\in C\}.
\]
For a finite set of vectors $A$, write
$\operatorname{pos}A=\{\sum_{a\in A}\lambda_a a:\lambda_a\geq0\}$ for its
positive hull.  For the principal vertex $p$ in the statement of the lemma,
the strict interior point supplied by Jung's theorem gives the standard
normal-cone formula
\begin{equation}\label{eq:ball-intersection-normal-cone}
 N_K(p)=\operatorname{pos}\{p-x_i:|p-x_i|=1\}.
\end{equation}
Remove the inessential centres of $X$; this does not change $K$ by
\cite[Theorem~5.1]{disk-polygons}.  Lemma~6.1 of \cite{disk-polygons} says
that, at a principal vertex, the rays from the vertex to the incident centres
are the extreme rays of an apexed cone.  There are at least three such rays,
so this cone is three-dimensional: a pointed cone of dimension at most two
has at most two extreme rays.  Equivalently, $N_K(p)$ has nonempty interior.

For $1\leq i,j\leq m$, put
\[
 C_{ij}=\operatorname{pos}\{p-x_i,p-x_j\}.
\]
The finite union of the two-generated cones $C_{ij}$ has empty interior in
$\mathbb R^3$.  Choose
\[
 u\in\operatorname{int}N_K(p)\setminus\bigcup_{i,j}C_{ij}.
\]
The intersection of finitely many balls with nonempty interior is strictly
convex.  Hence $p$ is the unique maximizer of
$z\mapsto\langle u,z\rangle$ on $K$.  Let $p_n$ be the corresponding unique
maximizer on $K_n$.  Hausdorff convergence gives $p_n\to p$.

We claim that $p_n$ is principal for all sufficiently large $n$.  Otherwise,
pass to a subsequence on which it is not principal.  The extended GHS theorem
\cite[Theorem~7.1]{disk-polygons} makes each $X^n$ tight, so at most two
generating facets are active at $p_n$.  The normal-cone formula for an
intersection of balls applies because Jung's theorem supplies a strict
feasible point.  After fixing the two labels along a further subsequence, it
gives
\[
 u=\alpha_n(p_n-x_i^n)+\beta_n(p_n-x_j^n),
 \qquad \alpha_n,\beta_n\geq0,
\]
where one coefficient may vanish.  The active normals have unit length, and
the diameter bound yields
\[
 \langle p_n-x_i^n,p_n-x_j^n\rangle
 =1-\frac12|x_i^n-x_j^n|^2\geq\frac12.
\]
Consequently $|u|^2\geq\alpha_n^2+\beta_n^2$, so the coefficients are
bounded.  Passing to the limit gives
$u\in\operatorname{pos}\{p-x_i,p-x_j\}=C_{ij}$, contrary to the choice of
$u$.  Thus $p_n$ is eventually principal.

Again by Theorem~7.1 of \cite{disk-polygons},
$\Ver B(X^n)=X^n$.  Hence $p_n=x_{k(n)}^n$; a fixed label occurs along a
subsequence, and its limit is $p$.  Therefore $p\in X$.
\end{proof}

For a finite unit-diameter set $A$, a centre $a\in A$ is called
\emph{essential} if
\[
 B(A)\subsetneq B(A\setminus\{a\}).
\]
We write $\operatorname{ess}(A)$ for the set of essential centres and call
$A$ \emph{tight} if $A=\operatorname{ess}(A)$.  We shall use the following
standard facts.  If a point of $A$ is incident with at least two diameters,
then it is essential, and
\begin{equation}\label{eq:essential-centres-ball}
 B(\operatorname{ess}(A))=B(A);
\end{equation}
see \cite[Lemma~3.2(ii),(iii)]{meissner_hynd}.  Moreover, if $A$ is tight
and has at least three points, every facet of $B(A)$ contains at least two
distinct principal vertices \cite[Proposition~6.1(i)]{disk-polygons}.

\begin{prop}[convergence of extremal sets]\label{prop:conv-extremal}
Let $m\geq4$ and let
$X^n=(x_1^n,\ldots,x_m^n)$ be extremal sets of unit diameter.  Suppose that
$x_i^n\to x_i$ for every $i$, and let $X$ be the set of distinct limit
points.  Let
\[
 Y=\{x\in X:\deg_{D(X)}(x)\geq2\}.
\]
Then $Y=\operatorname{core}(X)$, the set $Y$ has at least four points, and it
is an extremal set of unit diameter.  In particular,
\[
 \#E(D(Y))=2\#Y-2.
\]
Moreover, $B(X)=B(Y)$.
\end{prop}

\emph{Proof.}
We divide the argument into three steps.

\emph{Step 1: $Y$ has at least four points.}
There are only finitely many labelled graphs on $m$ vertices.  After passing
to a subsequence, which does not change $X$, we may consequently assume that
all $D(X^n)$ are the same graph $G$.  The strong self-duality theorem for
critical extremal diameter graphs
\cite[Theorem--Definition~9.1]{disk-polygons}, together with
\cite[Theorem~5.1]{meissner_graphs}, gives $\chi(G)=4$ in the critical case.
If $G$ has dangling vertices, delete its degree-two dangling vertices
successively.  Each deletion preserves extremality and leaves, after
finitely many steps, a critical extremal graph $G_0$, so $\chi(G_0)=4$.
Reinsert the deleted vertices in reverse order.  At the time of reinsertion
each has only its two former neighbours, and hence any four-colouring extends
to it.  Since $G_0\subset G$, this proves $\chi(G)=4$ in general.
Define $\phi(i)=x_i$.  If $ij\in E(G)$, then
$|x_i^n-x_j^n|=1$ for every $n$, and hence $|x_i-x_j|=1$.  In particular,
$x_i\neq x_j$, so $\phi$ is a graph homomorphism from $G$ to $D(X)$.  It
follows that $\chi(D(X))\geq\chi(G)=4$.

Choose a vertex-minimal subgraph $H\subset D(X)$ with $\chi(H)=4$.  Every
vertex of $H$ has degree at least three in $H$: otherwise, a $3$-colouring
after deletion of a vertex of degree at most two could be extended to $H$.
Thus every vertex of $H$ belongs to $Y$, and $\#Y\geq\#H\geq4$.

\emph{Step 2: $Y$ is precisely the set of essential centres of $X$.}
Put $E=\operatorname{ess}(X)$.  Every point of $Y$ is incident with at least
two diameters of $X$, so $Y\subset E$.  In particular, $E$ has at least four
points.  By \eqref{eq:essential-centres-ball},
\[
 B(E)=B(X),
\]
and $E$ is tight.  Indeed, if $q\in E$, then
\[
 B(E\setminus\{q\})\supset B(X\setminus\{q\})
 \supsetneq B(X)=B(E),
\]
so $q$ remains essential after the inessential centres of $X$ are removed.
The facet $F_q$ of $B(E)$ therefore contains two
distinct principal vertices $a,b$.  Each is at unit distance from at least
three centres of $E\subset X$, and hence is also a principal vertex of
$B(X)$.  By Lemma~\ref{lem:principal-stability}, $a,b\in X$.  Since
$a,b\in F_q$, we have
\[
 |q-a|=|q-b|=1.
\]
Thus $q$ has at least two neighbours in $D(X)$, so $q\in Y$.  This proves
$E\subset Y$, and consequently
\begin{equation}\label{eq:core-ball-equality}
 Y=\operatorname{ess}(X),\qquad B(Y)=B(X).
\end{equation}
In particular, $Y$ is tight.

\emph{Step 3: $Y$ is the vertex set of $B(Y)$ and is the $2$-core of
$D(X)$.}
Fix $y\in Y$.  Since $Y$ is tight and has at least four points, the facet
$F_y$ of $B(Y)$ contains two distinct principal vertices $a,b$.  Each is at
unit distance from at least three centres of $Y\subset X$, so each is also a
principal vertex of $B(X)$.  Lemma~\ref{lem:principal-stability} therefore
gives $a,b\in X$.  Their three neighbours in $Y$ show that they have degree
at least three in $D(X)$, and hence $a,b\in Y$.  Thus $y$ has at least two
distinct diameter neighbours in $Y$.  As $y$ is itself a centre of $B(Y)$,
it is a dangling or principal vertex according as it has exactly two or at
least three such neighbours.  This proves $Y\subset\Ver B(Y)$.

Conversely, a dangling vertex of $B(Y)$ belongs to $Y$ by definition.  If
$p$ is a principal vertex of $B(Y)=B(X)$, then
Lemma~\ref{lem:principal-stability} gives $p\in X$.  Since $p$ is at unit
distance from at least three centres in $Y$, it belongs to $Y$.  Hence
\begin{equation}\label{eq:core-is-vertex-set}
 Y=\Ver B(Y).
\end{equation}

It follows from the preceding argument that every vertex of $D(Y)$ has
degree at least two.  On the other hand, every point of $X\setminus Y$ has
degree at most one already in $D(X)$ and therefore cannot belong to its
graph-theoretic $2$-core.  Hence $Y=\operatorname{core}(X)$.

Finally, $D(Y)$ has minimum degree at least two, so $Y$ contains a pair at
distance one.  As $Y\subset X$ and $\diam X\leq1$, this gives $\diam Y=1$.
Now $\#Y\geq4$ and \eqref{eq:core-is-vertex-set} allow us to apply the
extended GHS theorem \cite[Theorem~7.1]{disk-polygons}: $Y$ has
$2\#Y-2$ diameters and is extremal.
\hfill$\square$

Proposition~\ref{prop:conv-extremal} is the geometric compactness statement
needed at a collision: an exact limit of extremal configurations remains in
the extremal class after merging coincident points and pruning its diameter
graph.  It does not say that the dual-edge incidences of the labelled graph
survive the collision; those must still be reconstructed from $Y$.

\subsection{Compact labelled relaxation and degenerate limits}

Let $\mathcal D$ be a labelled diameter graph on $\{1,\ldots,m\}$ and let
$\mathcal E$ be a fixed list of unordered dual-edge pairs.  We consider the
normalized feasible set
\begin{equation}\label{eq:fixed-feasible-set}
 \mathcal A(\mathcal D)=\left\{(x_1,\ldots,x_m)\in(\Bbb R^3)^m:
 \begin{array}{l}
 \sum_{i=1}^m x_i=0,\\[-2pt]
 |x_i-x_j|=1\quad ((i,j)\in\mathcal D),\\[-2pt]
 |x_i-x_j|\leq 1\quad (1\leq i<j\leq m)
 \end{array}\right\}.
\end{equation}
The centroid condition removes translations.  Rotations are deliberately not
factored out.
When a labelled summand is asymmetric, one representative ordering is fixed
for each member of $\mathcal E$; this is an additional choice and does not affect
compactness.

Inside this relaxation, let
$\mathcal A_{\rm geo}(\mathcal D,\mathcal E)$ be the subset of distinct
configurations whose complete unit-distance graph is exactly $\mathcal D$ and
whose exposed dual-edge list is exactly $\mathcal E$.  This subset need not be
closed.  Thus $\mathcal A(\mathcal D)$ is the compact labelled relaxation,
whereas optimization on $\mathcal A_{\rm geo}$ generally has only an a priori
supremum or infimum.

\begin{prop}\label{prop:fixed-graph-compact}
The set $\mathcal A(\mathcal D)$ is compact.  If it is nonempty, then for every
continuous $F:[0,1]^2\to\Bbb R$, the labelled functional
\begin{equation}\label{eq:fixed-labelled-functional}
 J_{\mathcal E}(X)=\sum_{(i,j,k,l)\in\mathcal E}
 F(|x_i-x_j|,|x_k-x_l|)
\end{equation}
attains its minimum and maximum on $\mathcal A(\mathcal D)$.
\end{prop}

\emph{Proof.}  The defining constraints are closed.  Moreover, for each $i$,
the centroid condition and the diameter bound give
\[
 |x_i|=\left|\frac1m\sum_{j=1}^m(x_i-x_j)\right|
 \leq \frac{m-1}{m}.
\]
Thus $\mathcal A(\mathcal D)$ is bounded and hence compact.  The assertion for
$J_{\mathcal E}$ follows from continuity. \hfill$\square$

Proposition~\ref{prop:fixed-graph-compact} concerns a fixed \emph{formal}
dual-pair list.  We now record the corresponding geometric limit statement,
where dual pairs are reconstructed after taking the diameter core.  The
extra exposure information is not determined merely by the four cross
distances in the dual-edge criterion.  We therefore separate persistence of
the dual pairs of the limit from the stronger requirement that all remaining
pairs degenerate.

\begin{deff}[Compatible and stable convergence of exposed dual pairs]
\label{def:stable-dual-pairs}
Let $X^n=(x_1^n,\ldots,x_N^n)$ be labelled extremal unit-diameter sets and,
after translations, suppose that $x_i^n\to z_i$ for every $i$.  Let $Z$ be
the set of distinct points among the $z_i$ and put
$X=\operatorname{core}(Z)$.  We say that the exposed dual pairs converge
\emph{compatibly} if, for all sufficiently large $n$, there are
injections
\[
 \iota_n:\mathcal E(X)\longrightarrow\mathcal E(X^n)
\]
such that, for every $(e,e')\in\mathcal E(X)$, the four endpoints of
 $\iota_n(e,e')$ converge to those of $(e,e')$, up to exchanging the two
 edges and the endpoints of either edge.
The convergence is called \emph{stable} if, in addition, the unmatched
pairs collapse uniformly, in the precise sense that
 \begin{equation}\label{eq:uniform-unmatched-collapse}
  \max_{(a,a')\in
  \mathcal E(X^n)\setminus\iota_n(\mathcal E(X))}
  \min\{L_{X^n}(a),L_{X^n}(a')\}\longrightarrow0,
 \end{equation}
 where the maximum of the empty set is understood to be zero.
Thus compatibility requires every dual pair of the limit to have a distinct
lift.  Stability additionally says that the nondegenerate antipodal
rectangle pairs in the linked partitions converge with multiplicity one to
the rectangle pairs of $\mathcal P(X)$, while every remaining rectangle pair
degenerates; the linked partitions are introduced in
Subsection~\ref{subsec:linked-partitions}.
\end{deff}

Let $q_X(e)\geq0$ be a size assigned to an exposed edge $e$ of $B(X)$.  In
the result below we assume that, for some $Q<\infty$,
\begin{equation}\label{eq:q-stability-assumptions}
 0\leq q_X(e)\leq Q,
\end{equation}
uniformly over the class under consideration; that $q$ is continuous along
endpointwise convergent nondegenerate exposed pairs; and that
$q_{X^n}(e_n)\to0$ whenever the endpoints of $e_n$ coalesce.  The Euclidean
chord length $L$, the spherical (geodesic) length $\theta$, and the
ball-polyhedron arc length $\ell$ satisfy these assumptions.  For an
extremal set $X$, define the intrinsic functional
\begin{equation}\label{eq:intrinsic-dual-functional}
 J_F^q(X)=\sum_{(e,e')\in\mathcal E(X)}
 F\bigl(q_X(e),q_X(e')\bigr),
\end{equation}
where $\mathcal E(X)$ is the reconstructed set of exposed dual-edge pairs.
In this intrinsic definition $F$ is assumed symmetric.  An asymmetric
summand requires an oriented dual-pair structure and a compatible rule for
transporting that orientation through a limit.

\begin{prop}[Functionals under compatible and stable convergence]
\label{prop:dual-functional-convergence}
Let $X^n\to Z$ be a convergent sequence for which the exposed dual pairs
converge compatibly, and put
\[
 X=\operatorname{core}(Z).
\]
Fix injections $\iota_n$ witnessing this compatibility.
Thus $X$ is the extremal set supplied by
Proposition~\ref{prop:conv-extremal}.  Suppose that $q$ satisfies
\eqref{eq:q-stability-assumptions} and the two continuity assumptions above,
and let $F:[0,Q]^2\to[0,\infty)$ be continuous and symmetric.
\begin{enumerate}[label=(\alph*),noitemsep]
 \item One has
 \begin{equation}\label{eq:dual-functional-lsc}
  J_F^q(X)\leq\liminf_{n\to\infty}J_F^q(X^n).
 \end{equation}
 \item If, in addition, the convergence is stable and
 \begin{equation}\label{eq:F-vanishing-axes}
  st=0\quad\Longrightarrow\quad F(s,t)=0,
 \end{equation}
 then
 \begin{equation}\label{eq:dual-functional-continuity}
  J_F^q(X^n)\longrightarrow J_F^q(X).
 \end{equation}
\end{enumerate}
\end{prop}

\begin{proof}
For every $(e,e')\in\mathcal E(X)$, write
$\iota_n(e,e')=(e_n,e_n')$, using the harmless exchanges allowed in
Definition~\ref{def:stable-dual-pairs}.  The matched terms satisfy
\[
 F\bigl(q_{X^n}(e_n),q_{X^n}(e_n')\bigr)
 \longrightarrow F\bigl(q_X(e),q_X(e')\bigr).
\]
The sum of the matched terms therefore converges to $J_F^q(X)$.  The
unmatched terms are nonnegative, so discarding them proves
\eqref{eq:dual-functional-lsc} directly.

Under stable convergence and \eqref{eq:F-vanishing-axes}, one size
in every unmatched pair tends uniformly to zero by
\eqref{eq:uniform-unmatched-collapse} and the collapse assumption on $q$,
while the other remains in $[0,Q]$.  Uniform continuity of $F$ on
$[0,Q]^2$, together with its vanishing on the coordinate axes, makes every
unmatched term tend to zero.  The number of pairs is fixed, and therefore the
whole sum converges, proving \eqref{eq:dual-functional-continuity}.
\end{proof}

\begin{rem}
The distinction between compatibility and stability is important.  If
$q_{X^n}(e_n)\to0$ and $q_{X^n}(e_n')\to t>0$, then the disappearing pair
contributes $F(0,t)$ before it is deleted from the intrinsic core cost.
Nonnegativity therefore gives lower semicontinuity, whereas continuity of the
functional across the collapse requires $F(0,t)=0$ (and symmetrically
$F(t,0)=0$).  More generally, several nondegenerate sequence pairs may have
the same limiting pair.  Compatibility allows this multiplicity, whereas
stability excludes it.  Such positive multiplicities may prevent continuity
even when $F$ vanishes on the coordinate axes.
\end{rem}

\begin{rem}[A possible non-stable collision]
The proved dual-edge theorem runs in one direction: an exposed edge has a
unique exposed dual edge, and the four cross distances equal one
\cite[Theorem~8.1]{disk-polygons}.  The converse needs exposure.  For example,
insert a dangling point $p_4$ in the relative interior of the
ball-polyhedron edge $p_2p_3$ of a regular tetrahedron, dual to $p_0p_1$.
All four cross distances between $\{p_0,p_1\}$ and $\{p_2,p_3\}$ remain one,
but $p_2p_3$ is no longer exposed: it is subdivided into $p_2p_4$ and
$p_4p_3$.  Thus pair exposure can be subdivided or remapped even when the
distance equalities persist.

A fixed-graph boundary computation for the database graph
\texttt{n11\_g0022} suggests a more substantial limiting mechanism.  Its
vertices approach four collision classes of sizes $3,3,2,3$, whose distinct
limits form a regular tetrahedron.  The diameter graph has no edge inside any
of the four classes.  Of the ten prescribed dual pairs, three acquire a
collapsing edge, while the other seven approach the three tetrahedral dual
pairs with multiplicities $3,1,3$.
After one lift of each limit pair has been chosen, four unmatched pairs
therefore have both Euclidean lengths approaching one.  The additional
labelled diameters appearing at the collision are precisely what permits this
positive multiplicity.  For the product of spherical lengths, their residual
contribution would approach $4(\pi/3)^2$.

This computation motivates the distinction in
Definition~\ref{def:stable-dual-pairs}, but it is not used as an exact
counterexample: an exact curve of distinct extremal embeddings approaching
the four-class collision has not yet been established.  Compatible
convergence is expected to follow from local persistence of the relative
interiors of nondegenerate exposed arcs, but this has not been proved here
for arbitrary collisions.  Stable convergence is stronger and rules out the
positive multiplicities described above.  Therefore Proposition
\ref{prop:dual-functional-convergence} does not imply existence of an
optimizer for the intrinsic functional over all extremal sets with at most
$m$ vertices.  The conclusion on existence remains valid for any class which
is closed under taking cores and for which every optimizing sequence has a
convergent subsequence with stable exposed dual pairs.  Proposition
\ref{prop:fixed-graph-compact} still gives existence for the compact
\emph{labelled} relaxation without additional assumptions.
\end{rem}

The compact set \eqref{eq:fixed-feasible-set} permits coincident labelled
vertices.  Along an exact convergent sequence of extremal configurations,
Proposition~\ref{prop:conv-extremal} identifies the reduced geometric limit as
an extremal set.  The formal list $\mathcal E$, however, need not be the
dual-edge list of its limiting ball polyhedron.  Moreover, a numerical solver
returns only an approximate feasible point, to which the exact proposition
cannot be applied without verification.  Consequently, when vertices
coalesce we merge coincident points, reconstruct all unit distances, prune to
the diameter core, and reconstruct the exposed dual edges.  A reduced
configuration with $m_0$ vertices is accepted only if it has $2m_0-2$
diameter edges and $m_0-1$ dual pairs.  Proposition
\ref{prop:fixed-graph-compact} proves existence for the labelled nonlinear
program; it does not identify its degenerate boundary with a fixed
combinatorial class or preserve the original dual-pair list.

The conclusions for the labelled and geometric problems should be kept
separate.  Proposition~\ref{prop:fixed-graph-compact} gives both a minimum
and a maximum for the continuous \emph{formal} cost on the compact labelled
set $\mathcal A(\mathcal D)$.  Such an optimizer may lie on a boundary where
labels coincide, and its prescribed list $\mathcal E$ need not be the
dual-edge list of the reconstructed limit.

For an intrinsic geometric cost, Proposition
\ref{prop:dual-functional-convergence}(a) gives lower semicontinuity when
$F$ is continuous and nonnegative and the dual pairs converge compatibly.
This is the conclusion relevant to a minimization problem: after
reconstruction, the value of the limiting core cannot exceed the lower limit
of the approximating values.  For a maximization problem, stable convergence
together with the axis-vanishing condition \eqref{eq:F-vanishing-axes} gives
full continuity by Proposition
\ref{prop:dual-functional-convergence}(b), and the disappearing dual pairs
contribute zero.

Neither conclusion makes $\mathcal A_{\rm geo}(\mathcal D,\mathcal E)$
closed or proves attainment within a fixed geometric class.  Accordingly,
the geometric problems below are written with an infimum or a supremum, as
appropriate.  At a collision the configuration and its intrinsic dual-pair
list must still be reconstructed.

\subsection{Linked spherical partitions and a conjectural converse}
\label{subsec:linked-partitions}

Let $X=(x_1,\ldots,x_m)$ be extremal.  For each $x_i$, intersect $\Bbb S^2$
with the cone generated by the diameter directions $x_j-x_i$ and denote the
resulting geodesically convex polygon by $P_i$.  The polygons $P_i$ and their
antipodes have pairwise disjoint interiors.  Their complement consists of
spherical rectangles indexed by the dual-edge pairs; see
\cite[Proposition~3.3]{bogosel_Meissner}.  Every boundary arc separates a
polygonal cell from a rectangular cell, with a zero-area polygonal cell
inserted when two rectangles have the same geometric edge.  If
$\mathcal P(X)$ denotes the complete partition, let $L(\mathcal P(X))$ be the
total one-dimensional boundary measure counted with rectangle-side incidence
multiplicity; thus coincident sides separated by a void cell are counted
separately.  Then
\begin{equation}\label{eq:partition-length}
 L(\mathcal P(X))=4\sum_{i=1}^{m-1}
 \bigl(\theta(e_i)+\theta(e_i')\bigr).
\end{equation}
The factor four records the antipodal rectangle and the two opposite sides of
each length.  The four vertices of a rectangular cell are also the vertices
of a Euclidean rectangle.  In particular, an interval cut out in a polygonal
cell by a great circle has length at most $\pi/3$, while such an interval in a
rectangular cell has length at most $\pi/2$.

There is a useful two-colouring hidden in this construction.  Write
$P_i^+=P_i$ for the face-normal polygon and $P_i^-=-P_i$ for the opposite
vertex-normal polygon.  The antipodal map exchanges the two colours.  The
four sides of each rectangle meet cells with colours $+,-,+,-$ in cyclic
order.  If $x_i$ is dangling, then $P_i^+$ and $P_i^-$ are geodesic arcs
rather than two-dimensional polygons.  Such an arc is retained as a marked
zero-area, or \emph{void}, cell.  Thus two rectangles are allowed to have the
same geometric edge, with the void cell recording the incidence between
them.

Following \cite[Proposition--Definition~2.1]{disk-polygons}, an extremal set
is called \emph{critical} if every vertex of its diameter graph has degree at
least three.  We now discuss the converse in this nondegenerate case, where
every polygonal cell is two-dimensional.  The void-cell case requires a
separate contraction and subdivision convention and is not part of the
statement below.

Call a cell decomposition of $\Bbb S^2$ a \emph{nondegenerate linked
two-coloured partition} if it has the following properties.
\begin{enumerate}[label=(\roman*),noitemsep]
 \item It consists of $m$ labelled antipodal polygon pairs
 $P_i^-=-P_i^+$ and $m-1$ antipodal pairs of spherical rectangles.  Rectangle
 sides meet polygonal cells with colours $+,-,+,-$ in cyclic order.  At each
 partition vertex exactly one cell of each colour and two rectangles meet;
 the two rectangles meet only at that vertex.
 \item Each antipodal pair of partition vertices is labelled
 $u_{ij}=-u_{ji}\in\Bbb S^2$, where $u_{ij}$ belongs to $P_i^+$ and
 $P_j^-$, while $u_{ji}$ belongs to $P_j^+$ and $P_i^-$.  These labels define
 a connected simple graph $D$ on $\{1,\ldots,m\}$, and
 \[
 P_i^+=\operatorname{sconv}\{u_{ij}:ij\in E(D)\}.
 \]
 After each antipodal polygon pair is contracted to its label, the incidence
 data define a cellular quadrangulation of $\Bbb{RP}^2$ whose vertices are the
 labels $i$, whose edges are the antipodal pairs $\{u_{ij},u_{ji}\}$, and
 whose faces are the antipodal rectangle pairs.
 \item A rectangle recording the prospective dual pair $(ij,k\ell)$ has,
 up to cyclic reversal, the corners
 \[
 u_{ki},\quad u_{kj},\quad u_{\ell j},\quad u_{\ell i},
 \]
 and these four points are the vertices of a Euclidean rectangle.  Every
 boundary arc has spherical length at most $\pi/3$.
\end{enumerate}
The Euler relation in $\Bbb{RP}^2$ gives $\#E(D)=2m-2$.  Thus this edge count
need not be imposed as a separate metric assumption.

\begin{prop}[automatic integration]\label{prop:partition-integration}
Every nondegenerate linked two-coloured partition determines points
$x_1,\ldots,x_m\in\Bbb R^3$, unique up to a common translation, such that
\begin{equation}\label{eq:partition-integrated-edges}
 x_j-x_i=u_{ij},\qquad |x_j-x_i|=1\quad(ij\in E(D)).
\end{equation}
In particular, the oriented vectors close around every cycle: if
$i_0i_1\ldots i_r=i_0$ is a cycle of $D$, then
\[
 \sum_{s=0}^{r-1}u_{i_si_{s+1}}=0.
\]
We refer to this property as \emph{cycle closing}; it is a consequence of the
rectangular partition and is not an additional hypothesis.  Moreover, if
$ij$ labels a side of one of the rectangles and that side has spherical
length $\alpha_{ij}$, then
\begin{equation}\label{eq:partition-side-chord}
 |x_i-x_j|=2\sin\frac{\alpha_{ij}}2\leq1.
\end{equation}
\end{prop}

\emph{Proof.}
For a rectangle with cyclic corners
$u_{ki},u_{kj},u_{\ell j},u_{\ell i}$, the Euclidean parallelogram identity
gives
\begin{equation}\label{eq:rectangle-cycle-closing}
 u_{ki}-u_{\ell i}+u_{\ell j}-u_{kj}=0.
\end{equation}
This is precisely the vector closing equation on the four-cycle
$k i\ell j k$ of $D$.

After contracting each antipodal polygon pair to its label, the incidence data
give a quadrangular cell decomposition of $\Bbb{RP}^2$ with $D$ as its
one-skeleton and one $2$-cell for each of the $m-1$ antipodal rectangle pairs.
With real coefficients,
$H_1(\Bbb{RP}^2)=H_2(\Bbb{RP}^2)=0$; see, for example,
\cite{HatcherAlgebraicTopology}.  Consequently the boundaries of these
$m-1$ rectangles form a basis of the cycle space of $D$: they span because
$H_1=0$, and they are independent because $H_2=0$.  Equation
\eqref{eq:rectangle-cycle-closing} therefore implies closing on every cycle
of $D$.  We refer to \cite{DiestelGraphTheory} for the graph-theoretic cycle
space.

Fix $x_1=0$ and define $x_j$ by summing the vectors $u_{ab}$ along a path in
$D$ from $1$ to $j$.  The cycle equations make the result independent of the
path and give \eqref{eq:partition-integrated-edges}.  Finally, two consecutive
corners $u_{ki},u_{kj}$ of a rectangle satisfy
\[
 |u_{ki}-u_{kj}|=|(x_i-x_k)-(x_j-x_k)|=|x_i-x_j|.
\]
The chord of a spherical arc of length $\alpha_{ij}$ has length
$2\sin(\alpha_{ij}/2)$, proving \eqref{eq:partition-side-chord}.
\hfill$\square$

The preceding proposition reconstructs all prescribed diameters and proves
the diameter bound for every pair appearing as a rectangle side.  What
remains is a global, rather than a cycle-closing, issue.

\begin{conj}[partition realization]\label{conj:partition-realization}
The points reconstructed from a nondegenerate linked two-coloured partition
are distinct and satisfy
\begin{equation}\label{eq:partition-global-diameter}
 |x_i-x_j|\leq1\qquad(1\leq i<j\leq m).
\end{equation}
Equivalently, the non-overlap and completeness of the spherical partition,
together with the bound $\alpha\leq\pi/3$ on its boundary arcs, force the
global diameter bound for pairs which are neither prescribed diameters nor
rectangle sides.
\end{conj}

This is the remaining metric implication in the proposed converse.  If
Conjecture~\ref{conj:partition-realization} holds, then the graph $D$ supplies
$2m-2$ unit distances among $m$ distinct points of diameter at most one.
The V\'azsonyi theorem makes the reconstructed set extremal and also rules out
any additional unit distance.  The polygon and rectangle incidences then
reproduce the original partition.  Hence, subject to this conjecture,
nondegenerate extremal unit-diameter sets modulo rigid motions are equivalent
to nondegenerate linked two-coloured partitions modulo orthogonal
transformations.  The forward implication is the construction of
\cite[Proposition~3.3]{bogosel_Meissner}.

\section{Additive and product dual-edge functionals}

\subsection{Minimization of additive functionals}

\begin{thm}
	\label{thm:minimal-length}
	a) Among all extremal finite sets of unit diameter, a regular tetrahedron is
	a minimizer of the total spherical length of dual edges
	\[ \sum_{i=1}^{m-1} \bigl(\theta(e_i)+\theta(e_i')\bigr).\]
		
	b) If $f:[0,\pi/3]\to[0,\infty)$ is concave and $f(0)=0$, then
		\[ \sum_{i=1}^{m-1} \bigl(f(\theta(e_i))+f(\theta(e_i'))\bigr)\]
		has a regular tetrahedron as a minimizer.
\end{thm}

\emph{Proof.} a) We first remove the void-cell case.  By the standard
dangling-edge description of extremal ball polyhedra
\cite{disk-polygons}, a vertex of degree two in the diameter graph can be
deleted by reversing
the subdivision described in Subsection~\ref{subsec:dangling-subdivision}, and the remaining set is again
extremal.  If the subdivision splits $e'$ into $e'_1,e'_2$, the total
spherical length changes by
\[
 \theta(e)+\theta(e'_1)+\theta(e'_2)-\theta(e')\geq0,
\]
where the inequality is the spherical triangle inequality.  Iterating the
deletion therefore cannot increase the objective.  It is enough to prove the
bound for an extremal set without dangling vertices.

Use the spherical partition recalled in the previous section for this reduced
set.  It contains each polygon $P_i$ and its antipode, and for each dual pair
$(e_i,e_i')$ it contains the associated spherical rectangle $R_i$ and its
antipode.  All polygonal cells are two-dimensional, so the partition has no
coincident boundary arcs separated only by a void cell.

If $u_{ij},u_{ik}$ are generators of $P_i$, then
$|u_{ij}-u_{ik}|=|x_j-x_k|\leq1$, and hence
$u_{ij}\cdot u_{ik}\geq1/2$.  Normalized positive combinations preserve
this lower bound, so the spherical diameter of $P_i$ is at most $\pi/3$.
The vertices of the rectangle $R_i$ form a Euclidean rectangle with side
lengths at most one.  Its diagonal therefore has length at most $\sqrt2$,
and all scalar products between its vertices are nonnegative.  The same
positive-combination argument shows that the spherical diameter of $R_i$ is
at most $\pi/2$.

Consider a great circle $C$ of $\Bbb{S}^2$ transverse to the partition
one-skeleton and avoiding its vertices and tangencies; the excluded circles
form a null set.  Along $C$, polygonal and rectangular cells alternate, and
the intersections occur in antipodal pairs.  Each cell has spherical diameter
strictly smaller than $\pi$, so its intersection with $C$ is either empty or
one interval.

We claim that $C$ intersects at least two pairs of antipodal rectangles.  If
only one pair were met, alternation would give only one pair of polygonal
intervals.  The total length covered would then be at most
$2(\pi/2)+2(\pi/3)<2\pi$, a contradiction.

Each rectangular interval has an entry and an exit point.  Hence $C$ meets the
boundary of the partition at least eight times.

The Cauchy-Crofton formula (see \cite{Santalo_Length} for example) allows us to compute the length of a family of curves on the sphere by counting intersections with great circles: the length of the partition $\mathcal P$ described above is
\[ L(\mathcal P) = \frac{1}{4} \int_{\Bbb S^2} \# (\mathcal P \cap \xi^\perp) d A(\xi),\]
where $\xi^\perp$ is the great circle orthogonal to $\xi\in\Bbb S^2$.
Since the intersection count is at least eight almost everywhere, it follows
that
\[ L(\mathcal P) \geq 8\pi,\]
the length of four great circles.  Combining this inequality with
\eqref{eq:partition-length} gives
\[
 \sum_{i=1}^{m-1}\bigl(\theta(e_i)+\theta(e_i')\bigr)\geq 2\pi.
\]
For the regular tetrahedron the six spherical edge lengths equal $\pi/3$, so
equality holds.

b) Since $f$ is concave, non-negative and $f(0)=0$ we have $f(t) \geq t/(\pi/3)f(\pi/3)$. Therefore
\[ \sum_{i=1}^{m-1} \bigl(f(\theta(e_i))+f(\theta(e_i'))\bigr) \geq \frac{f(\pi/3)}{\pi/3} \sum_{i=1}^{m-1} \bigl(\theta(e_i)+\theta(e_i')\bigr) \geq 6f(\pi/3), \]
where we used the inequality found at the previous point. The lower bound obtained corresponds to the value given by the regular tetrahedron: $6$ edges having length $\pi/3$. \hfill $\square$ 

\subsection{Dangling vertices and subdivision of dual pairs}
\label{subsec:dangling-subdivision}

There is a simple local operation on extremal sets that is useful in comparing
the functionals considered below.  Let $X$ be extremal, let $(e,e')$ be an
exposed dual pair, write $e=(x_i,x_j)$ and $e'=\arc{x_kx_l}$, and choose a
point $x_0$ in the relative interior of $e'$.  Since $e'\subset B(X)$ and
$|x_0-x_i|=|x_0-x_j|=1$, the set
\[
 X^+=X\cup\{x_0\}
\]
still has diameter one and has precisely two additional diameter edges.
Thus $X^+$ is extremal and $x_0$ is a dangling vertex of degree two.  At the
level of exposed edge incidences, the operation replaces
\begin{equation}\label{eq:dangling-replacement}
 (e,e')\quad\hbox{by}\quad (e_1,e'_1),(e_2,e'_2),
 \qquad
 e'_1=\arc{x_kx_0},\quad e'_2=\arc{x_0x_l},
\end{equation}
where $e_1,e_2$ are two combinatorial incidences of the same geometric edge
$e$, so every geometric edge size satisfies $q(e_1)=q(e_2)=q(e)$.  The two
copies record the sides of the digonal, or void, face created by $x_0$; this
incidence multiplicity does not assert that one edge belongs to two ordinary
dual pairs.  All other pairs are unchanged.  Deleting $x_0$ is the inverse
operation.  Here and below, ``deleting a dangling vertex'' refers to this
inverse exposed-edge operation; after deletion the intrinsic dual-pair list is
reconstructed, rather than retaining two degenerate labelled pairs.

\begin{prop}[Change under a dangling subdivision]\label{prop:dangling-subdivision}
Let $q$ be an edge size for which
\[
 q(e')\leq q(e'_1)+q(e'_2).
\]
For $J_F^q(X)=\sum_{(a,a')\in\mathcal E(X)}F(q(a),q(a'))$, the
insertion in \eqref{eq:dangling-replacement} changes the objective by
\begin{equation}\label{eq:dangling-change}
 J_F^q(X^+)-J_F^q(X)
 =F\bigl(q(e_1),q(e'_1)\bigr)+F\bigl(q(e_2),q(e'_2)\bigr)
  -F\bigl(q(e),q(e')\bigr).
\end{equation}
Consequently, insertion does not decrease the objective whenever, for every
fixed $s$, the function $t\mapsto F(s,t)$ is nondecreasing and subadditive.
Under the same hypothesis, deletion does not increase the objective.
\end{prop}

\begin{proof}
Formula \eqref{eq:dangling-change} follows because only the pair in
\eqref{eq:dangling-replacement} changes.  Put
$s=q(e)=q(e_1)=q(e_2)$.  Monotonicity and subadditivity give
\[
 F\bigl(s,q(e')\bigr)
 \leq F\bigl(s,q(e'_1)+q(e'_2)\bigr)
 \leq F\bigl(s,q(e'_1)\bigr)+F\bigl(s,q(e'_2)\bigr).
\]
The statement about deletion is the same comparison in reverse.
\end{proof}

Both $q=L$ and $q=\theta$, where
$\theta(a)=2\arcsin(L(a)/2)$, satisfy the triangle inequality required in the
proposition.  In the spherical case the three relevant endpoints lie on a
unit sphere centred at either endpoint of $e$, and the inequality is the
spherical triangle inequality.  Thus Euclidean products also do not decrease
under insertion.  For the spherical product, write $J_{\rm prod}^{\theta}$
for the corresponding intrinsic functional.  Then one has the especially
relevant identity
\begin{equation}\label{eq:dangling-spherical-product}
 J_{\rm prod}^{\theta}(X^+)-J_{\rm prod}^{\theta}(X)
 =\theta(e)\bigl(\theta(e'_1)+\theta(e'_2)-\theta(e')\bigr)
 \geq0.
\end{equation}
Thus adding dangling vertices cannot lower the sum of products of spherical
lengths of dual edges, while deleting them cannot raise it.  Repeatedly
subdividing $e'$ into finer exposed subarcs makes the sum of the spherical
endpoint distances converge to its circular arc length $\ell(e')$; hence the
subdivided contribution converges to $\theta(e)\ell(e')$.

There is no analogous unconditional rule for a general $F$.  For example, an
additively separated summand $F(s,t)=f(s)+f(t)$ also duplicates the term
$f(s)$ when a dangling point is inserted, and its sign cannot be inferred
without further assumptions on $f$.  The rectangle-area summand introduced
below is
\begin{equation}\label{eq:rectangle-area-summand}
 R(s,t)=4\arcsin\left(
 \frac{s}{\sqrt{4-s^2}}\frac{t}{\sqrt{4-t^2}}
 \right).
\end{equation}
It has the opposite behavior under subdivision from the products, but this
follows from the circular geometry rather than from the triangle inequality.

\begin{prop}[Rectangle area under dangling subdivision]
\label{prop:dangling-rectangle}
For the replacement \eqref{eq:dangling-replacement},
\[
 R(L(e),L(e'))\geq
 R(L(e),L(e'_1))+R(L(e),L(e'_2)).
\]
The inequality is strict when the inserted point lies in the relative
interior of $e'$.  Consequently, deleting a dangling vertex strictly
increases the sum of the areas of the linked spherical rectangles.
\end{prop}

\begin{proof}
Put $s=L(e)$ and let $\varphi$ be the central angle traced by $e'$ on the
circle $\partial B(x_i)\cap\partial B(x_j)$.  This circle has radius
$r=\sqrt{1-s^2/4}$, and hence
\[
 L(e')=2r\sin(\varphi/2).
\]
If $c=s/2$, substitution in the formula for $R$ gives
\[
 R\bigl(s,2r\sin(\varphi/2)\bigr)
 =4\arctan\bigl(c\tan(\varphi/2)\bigr)=:h_s(\varphi).
\]
Here $0\leq c\leq1/2$.  The dihedral-angle formula preceding
\eqref{eq:length-edge}, together with
$\theta(e),\theta(e')\leq\pi/3$, gives
$0\leq\varphi\leq2\arcsin(1/\sqrt3)<\pi$.  A direct differentiation gives
\[
 h_s'(\varphi)=
 \frac{2c}{\cos^2(\varphi/2)+c^2\sin^2(\varphi/2)},
\]
which is nondecreasing.  Thus $h_s$ is convex and $h_s(0)=0$, so
$h_s(u+v)\geq h_s(u)+h_s(v)$ for $u,v\geq0$ with $u+v<\pi$.
The two subarc angles add to $\varphi$, proving the claim.  For an exposed
dual pair $s=L(e)>0$.  If the inserted point is in the relative interior of
$e'$, both subarc angles are positive and $h_s$ is strictly convex on their
sum.  The inequality is then strict.
\end{proof}

\subsection{Product-type maximization problems}

The area formula for Meissner polyhedra \eqref{eq:area-Meiss-poly} involves
sums of terms depending on dual-edge pairs.  For a fixed labelled graph and a
fixed formal pair list, the relaxed problem on $\mathcal A(\mathcal D)$ has a
maximizer by Proposition \ref{prop:fixed-graph-compact}.  The geometric
problem is instead posed on $\mathcal A_{\rm geo}(\mathcal D,\mathcal E)$: a
boundary maximizer of the relaxation may contain coincident labels, and its
formal value is geometric only after reconstruction.

Let $\mathfrak X_{\rm ext}$ denote the class of finite extremal unit-diameter
sets, and put
\[
 \mathfrak R=
 \{(\mathcal D(X),\mathcal E(X)):X\in\mathfrak X_{\rm ext}\}.
\]
Thus abstract data $(\mathcal D,\mathcal E)$ are \emph{geometrically
realizable} precisely when $(\mathcal D,\mathcal E)\in\mathfrak R$, or,
equivalently, when
$\mathcal A_{\rm geo}(\mathcal D,\mathcal E)\ne\varnothing$.  The qualifier
is essential when the optimization is indexed by abstract graphs and pair
lists, since not every such list is induced by a configuration in
$\Bbb R^3$.  It is redundant when the outer supremum is written directly over
$X\in\mathfrak X_{\rm ext}$.  We use $\mathfrak R$ below to avoid repeating
the qualifier in every problem.

\subsubsection{Maximal products of lengths of dual edges.}

Fix a diameter graph $\mathcal D$ associated to an extremal set and let
$\mathcal E$ be its set of dual-edge pair indices.  Taking $F(s,t)=st$ in
\eqref{eq:general-functional} gives the problem
\begin{equation}
	 \mathcal P_1(\mathcal D,\mathcal E) = \sup_{X\in\mathcal A_{\rm geo}(\mathcal D,\mathcal E)} \sum_{(i,j,k,l)\in \mathcal E} |x_i-x_j|\cdot |x_k-x_l|
	 \label{eq:max-prod}
\end{equation}
This is the problem for a fixed realizable graph and dual-pair list.  Its
global value is
\[
 \mathfrak P_1:=
 \sup_{X\in\mathfrak X_{\rm ext}}
 \sum_{(e,e')\in\mathcal E(X)}L_X(e)L_X(e')
 =\sup_{(\mathcal D,\mathcal E)\in\mathfrak R}
 \mathcal P_1(\mathcal D,\mathcal E).
\]

First, let $X$ be extremal and set $K=\operatorname{conv}X$.  The facets of
$K+(-K)$ consist of a copy of each facet of $K$, its antipode, and an
antipodal pair of parallelograms $e-e'$ and $e'-e$ for every dual pair
$(e,e')$; see Figure~\ref{fig:flat-rectangles}.  To compute their areas,
write $e=[x_i,x_j]$ and $e'=[x_k,x_l]$.  The four cross-distances are one.
Subtracting the equalities
$|x_j-x_k|^2=|x_i-x_k|^2$ and
$|x_j-x_l|^2=|x_i-x_l|^2$ gives
$(x_j-x_i)\cdot(x_l-x_k)=0$.  Thus the parallelogram $e-e'$ has area
$L(e)L(e')$.  The same decomposition, with the natural subdivision of a
rectangle when a dangling vertex is present, gives
\[
 \Area(K+(-K))=2\Area(K)+
 2\sum_{(e,e')\in\mathcal E(X)}L(e)L(e').
\]
Consequently,
\[ \sum_{(i,j,k,l)\in \mathcal E} |x_i-x_j|\cdot |x_k-x_l| = \frac{1}{2}\left(\Area(K+(-K))-2\Area(K)\right) = A(K,-K),\]
where $A(K_1,K_2)$ is the mixed surface area in the normalization
\[
A(K_1+K_2)=A(K_1)+2A(K_1,K_2)+A(K_2);
\]
see \cite[Chapter~5]{schneider}.
Let $\mathcal{CW}_1$ be the class of convex bodies of constant width one and
write
\[
 \mathfrak B:=\sup_{W\in\mathcal{CW}_1}
 \bigl(2\pi-|\partial W|\bigr).
\]

\begin{prop}[Euclidean-product upper bound]
\label{prop:euclidean-product-upper}
One has
\begin{equation}\label{eq:P1-upper-bound}
 \mathfrak P_1\leq\mathfrak B.
\end{equation}
\end{prop}

\begin{proof}
Let $K'$ be a constant-width-one completion of
$K=\operatorname{conv}X$; for example, the Meissner construction
\eqref{eq:meiss-poly} supplies such a completion when $X$ is extremal.
Thus $K\subset K'$, and monotonicity of mixed surface area in both arguments
gives
\[
 \sum_{(e,e')\in\mathcal E(X)}L_X(e)L_X(e')
 =A(K,-K)\leq A(K',-K').
\]
Since $K'+(-K')$ is the unit ball, the normalization above gives
\[
 A(K',-K')=2\pi-|\partial K'|\leq\mathfrak B.
\]
Thus, among constant-width-one completions, the mixed surface area
$A(K',-K')$ is maximized by a surface-area minimizer.  Taking the supremum
over $X$ proves \eqref{eq:P1-upper-bound}.
\end{proof}

\begin{rem}[Open centre-hull density question]
It is not currently justified that equality holds in
\eqref{eq:P1-upper-bound}.  The completion argument above shows that every
finite extremal hull is bounded by the best constant-width completion; it
does not show that a surface-area-minimizing constant-width body is itself a
finite extremal hull, or is approximated in the required mixed-area sense by
such hulls.  Equality would follow, for example, if every
$W\in\mathcal{CW}_1$ admitted extremal sets $X_n$ with
$\operatorname{conv}X_n\to W$ in the Hausdorff metric; continuity of mixed
surface area would then give the reverse inequality.  The density theorem of
\cite{meissner_hynd} proves instead that Meissner bodies $M_n$ based on
extremal sets can converge to $W$.  It does not assert that the centre hulls
$\operatorname{conv}X_n$ converge to $W$.  Thus the Euclidean-product
equality, or an adequate weaker mixed-area approximation statement, remains
an open question.  This distinction is also consistent with the
three-dimensional relaxation in \cite{Bogosel_Mixed}.
\end{rem}
\begin{figure}
	\centering 
	\includegraphics[height=0.3\textwidth]{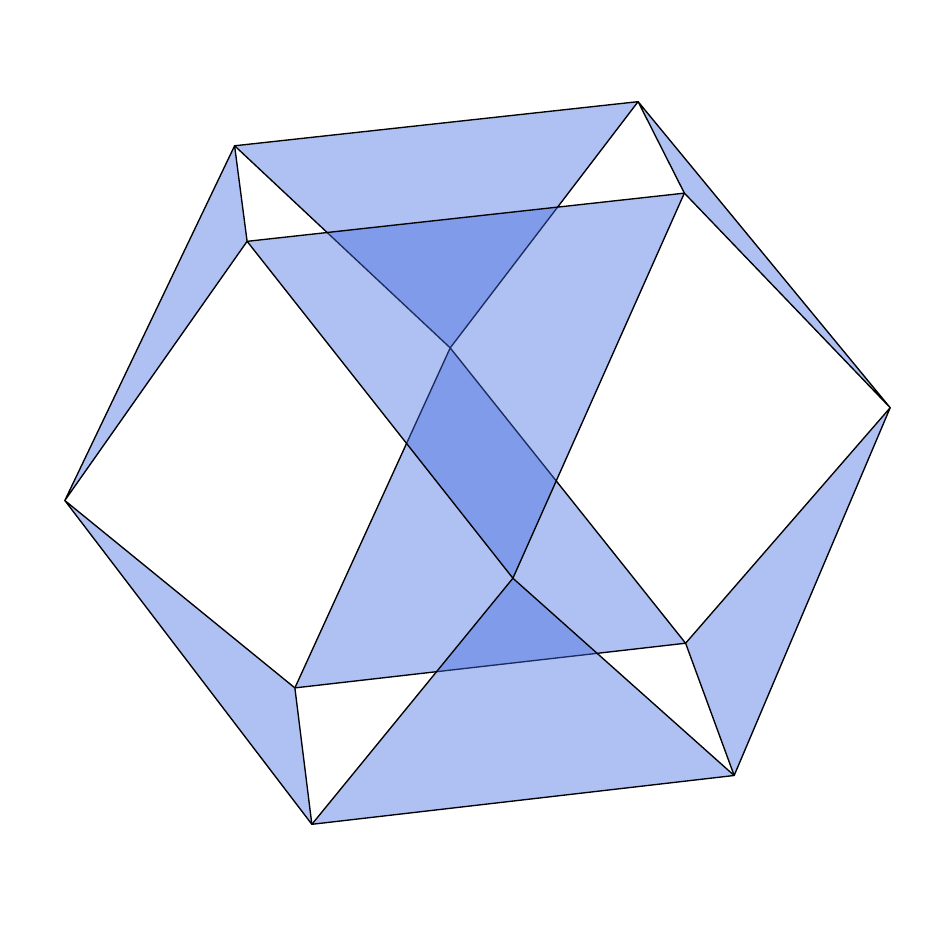}
	\includegraphics[height=0.3\textwidth]{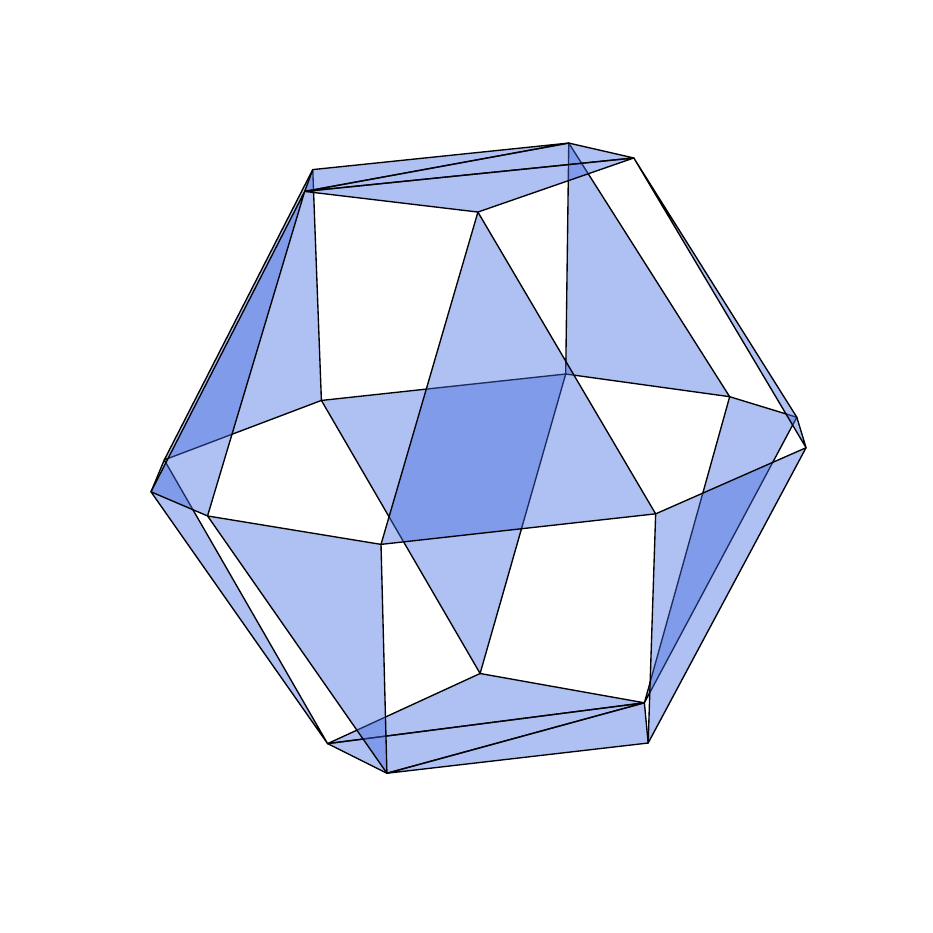}
	\includegraphics[height=0.3\textwidth]{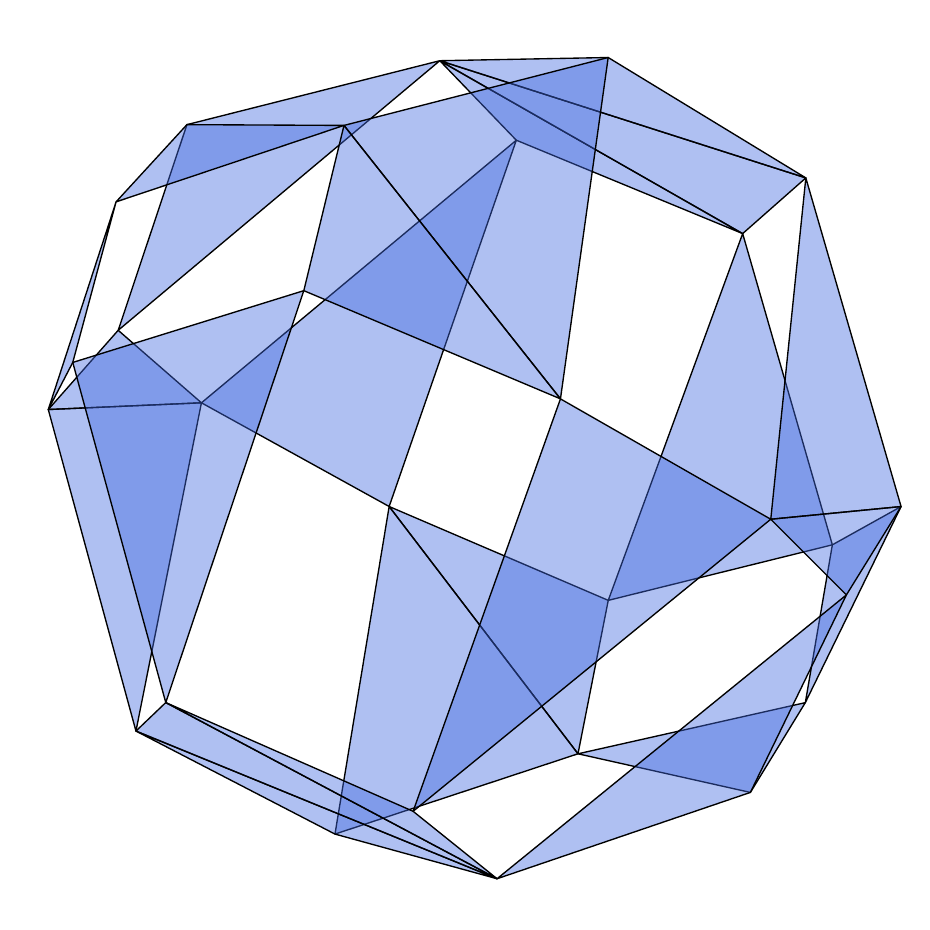}	
	\caption{Minkowski sums of extremal unit-diameter sets and their reflections through the origin. Rectangles associated with dual-edge pairs are visible.}
	\label{fig:flat-rectangles}
\end{figure}
\subsubsection{Maximal products of spherical lengths of dual edges}
\label{sec:max-prod-spherical}

The spherical length of a chord given by $x,y$ on the unit sphere is $2 \arcsin (|x-y|/2)$. Consider the problem
\begin{equation}
	\mathcal P_2(\mathcal D,\mathcal E) = \sup_{X\in\mathcal A_{\rm geo}(\mathcal D,\mathcal E)} \sum_{(i,j,k,l)\in \mathcal E} 4\arcsin \frac{|x_i-x_j|}{2} \cdot \arcsin \frac{|x_k-x_l|}{2}.
	\label{eq:max-spherical-prod}
\end{equation}
Its global value is
\[
 \mathfrak P_2:=
 \sup_{X\in\mathfrak X_{\rm ext}}
 \sum_{(e,e')\in\mathcal E(X)}\theta_X(e)\theta_X(e')
 =\sup_{(\mathcal D,\mathcal E)\in\mathfrak R}
 \mathcal P_2(\mathcal D,\mathcal E).
\]
Consider a dual pair $(x_i,x_j),(x_k,x_l)$, denoted $e,e'$. Since the geodesic edge is the shortest path on the unit sphere we have
\begin{equation}
	 \theta(e)\theta(e') \leq \theta(e)\ell(e'),
	 \label{eq:ineq-products}
\end{equation}
with $\ell$ defined in \eqref{eq:length-edge}.
Proposition~\ref{prop:dangling-subdivision} shows that repeatedly subdividing
$e'$ by dangling points does not decrease the spherical-product objective and
makes the sum of its spherical endpoint distances converge to $\ell(e')$.

\begin{thm}[Exact spherical-product reformulation]
\label{thm:spherical-product-reformulation}
One has
\[
 \mathfrak P_2=\mathfrak B.
\]
\end{thm}

\begin{proof}
For the upper bound, orient each dual pair so that the corresponding choice
of Meissner smoothing contributes $\theta(e)\ell(e')$ to
$2\pi-|\partial M|$ in \eqref{eq:area-Meiss-poly}.  Summing
\eqref{eq:ineq-products} gives the pointwise estimate
$J_{\rm prod}^{\theta}(X)\leq2\pi-|\partial M|\leq\mathfrak B$; taking
suprema gives $\mathfrak P_2\leq\mathfrak B$.

For the reverse inequality, let $\varepsilon>0$.  Choose a body $W$ of
constant width one such that
$2\pi-|\partial W|>\mathfrak B-\varepsilon/3$.  By the density of Meissner
polyhedra \cite{meissner_hynd} and continuity of surface area under Hausdorff
convergence, there is a Meissner polyhedron $M$ with
\[
 2\pi-|\partial M|>\mathfrak B-2\varepsilon/3.
\]
Subdivide each selected arc $e'$ of $M$ by finitely many dangling points,
with all subdivisions sufficiently fine that the resulting
spherical-product sum is larger than
$2\pi-|\partial M|-\varepsilon/3$.  The enlarged set has diameter at most
one by the Meissner construction, and every new centre contributes the two
diameters joining it to the endpoints of the dual edge $e$.  The V\'azsonyi
bound therefore makes it extremal.  Hence
$\mathfrak P_2>\mathfrak B-\varepsilon$.  Since $\varepsilon$ is arbitrary,
$\mathfrak P_2\geq\mathfrak B$, which proves the equality.
\end{proof}

In the notation above, therefore,
\begin{equation}\label{eq:P1-P2-comparison}
 \mathfrak P_1\leq\mathfrak P_2=\mathfrak B.
\end{equation}
The spherical-product problem is consequently an exact discrete supremal
reformulation of the constant-width area-deficit problem.  Whether the first
inequality in \eqref{eq:P1-P2-comparison} is an equality is the Euclidean
mixed-area approximation question isolated after Proposition
\ref{prop:euclidean-product-upper}; centre-hull density would be a sufficient
answer.

\subsubsection{Maximal areas of linked rectangle partitions on the sphere}
The function $R$ defined in \eqref{eq:rectangle-area-summand} is the area of the spherical rectangle
determined by a Euclidean rectangle with side lengths $s,t$ and vertices on
the unit sphere; see \cite[Lemma~1]{bogosel_Meissner}.  A natural question is
to consider the problem
\begin{equation}
	\mathcal P_3(\mathcal D,\mathcal E) = \sup_{X\in\mathcal A_{\rm geo}(\mathcal D,\mathcal E)} \sum_{(i,j,k,l)\in \mathcal E} R(|x_i-x_j|,|x_k-x_l|).
	\label{eq:max-rectangles}
\end{equation}
For an extremal set $X$, write
\begin{equation}\label{eq:rectangle-objective}
 J_{\rm rect}(X)=
 \sum_{(e,e')\in\mathcal E(X)}R(L_X(e),L_X(e')).
\end{equation}
Its global value is
\[
 \mathfrak P_3:=
 \sup_{X\in\mathfrak X_{\rm ext}}J_{\rm rect}(X)
 =\sup_{(\mathcal D,\mathcal E)\in\mathfrak R}
 \mathcal P_3(\mathcal D,\mathcal E).
\]

\medskip
\noindent\emph{Extremal-inclusion comparison.}
If $Y\subsetneq X$ are critical extremal sets, then
\begin{equation}\label{eq:critical-inclusion-rectangle}
 J_{\rm rect}(X)<J_{\rm rect}(Y).
\end{equation}
Indeed, the points of $X\setminus Y$ lie between the wedges and the
corresponding spindles determined by $Y$.  More precisely, write
$P_i^\pm(Z)$ for the two polygonal cells associated with $x_i$ in the linked
partition of an extremal set $Z$.  For $x_i\in Y$, every diameter neighbour
of $x_i$ in $Y$ is also a diameter neighbour in $X$, and therefore
\[
 P_i^+(Y)\subset P_i^+(X),\qquad
 P_i^-(Y)\subset P_i^-(X).
\]
Thus the union of the polygonal cells for $Y$ is contained in the union of
those for $X$, so the union of the rectangular cells for $X$ is contained in
the union of those for $Y$.  The inclusion is strict in area: if
$x\in X\setminus Y$, criticality of $X$ makes $P_x^+$ and $P_x^-$
two-dimensional, and their interiors lie in the rectangular region of the
$Y$-partition.  The total area of the rectangular region is twice
$J_{\rm rect}$.  Hence $2J_{\rm rect}(X)<2J_{\rm rect}(Y)$, proving
\eqref{eq:critical-inclusion-rectangle}.

We can now state the stronger consequence.  By Proposition
\ref{prop:dangling-rectangle}, deleting a dangling vertex strictly increases
$J_{\rm rect}$.  Hence a maximizer, if one exists, is critical.  If such a
critical maximizer $X$ contained a proper extremal subset, delete the
dangling vertices of that subset until a critical extremal set
$Y\subsetneq X$ remains.  The deletions increase the rectangle objective,
and \eqref{eq:critical-inclusion-rectangle} gives
$J_{\rm rect}(Y)>J_{\rm rect}(X)$, a contradiction.  Therefore every
maximizer contains no proper extremal subset.  The same reduction applied to
an arbitrary extremal set also gives
\[
 \mathfrak P_3=
 \sup_{\substack{X\in\mathfrak X_{\rm ext}\\
                  X\text{ contains no proper extremal subset}}}
 J_{\rm rect}(X).
\]
This equality concerns suprema and does not assert that the unrestricted
global problem has a maximizer.

\section{Numerical simulations and open problems}
\label{sec:numerics}

\subsection{Database and direct-evaluation conjectures}

This section reports direct evaluations and numerical optimization results.
Since the computations are not certified, terms such as maximizer, minimizer,
and numerical optimizer refer to candidates found by the numerical methods.

We use all 1002 embeddings with at most 14 vertices from the database
associated with \cite{meissner_graphs}:
\begin{center}
 \href{https://github.com/mraggi/ReuleauxPolyhedra}{\nolinkurl{https://github.com/mraggi/ReuleauxPolyhedra}}.
\end{center}
For $m=15,16$ we use, in the same native graph and embedding formats, all
2,195 strong-involution configurations with 15 vertices and all 7,447 with
16 vertices.  Thus the combined direct-evaluation database has 10,644
records.  The repository for the $m=15,16$ data is
\begin{center}
 \href{https://github.com/beniamin-bogosel/ReuleauxPolyhedraData15-16}%
 {\nolinkurl{https://github.com/beniamin-bogosel/ReuleauxPolyhedraData15-16}}.
\end{center}
It contains the self-dual planar-code inputs, the strong-involution outputs,
validated embeddings, checksums, and reproduction scripts.  In contrast to
the original differential-evolution penalty, these embeddings were computed
by sparse gradient-based constrained optimization with an explicit clearance
variable that repels distinct vertices from coincidence.  Unit edges are
equality constraints, nonedges remain strictly below the unit diameter, and
the serialized output is accepted only after the diameter graph has been
reconstructed independently.  This distinction removes the arbitrary
positive lower threshold on nonedge lengths while still excluding merged
vertices.

The graph and validated-embedding records from this repository supply the
fixed constraints and the first initialization for the optimizations in
Subsection~4.2.

Code and scripts for reproducing the numerical results in this section are
available in the following GitHub repository:

\begin{center}
	\href{https://github.com/beniamin-bogosel/ExtremalDiameterGraphs3D}%
	{\nolinkurl{https://github.com/beniamin-bogosel/ExtremalDiameterGraphs3D}}.
\end{center}

For each supplied embedding $X$ we verify the $2m-2$ unit distances,
reconstruct the $m-1$ exposed dual pairs, and directly evaluate
\[
 S_F(X)=\sum_{(e,e')\in\mathcal E(X)}F(L(e),L(e')).
\]
No vertices are moved and no optimization is performed in this first
experiment.  The purpose is to identify inequalities suggested by the
database and, in particular, the objectives for which the regular tetrahedron
$K_4$ appears at one end of the observed range.

The Python and MATLAB implementations contain partly overlapping lists of
objectives.  In Table~\ref{tab:direct-objectives} we combine the two lists and
use the abbreviations
\[
 \Theta(t)=2\arcsin(t/2),\qquad
 \rho(t)=\frac{t}{\sqrt{4-t^2}},
\]
\[
 F_M(s,t)=4\arcsin(t/2)\sqrt{1-t^2/4}
 \arcsin\left(\frac{s/2}{\sqrt{1-t^2/4}}\right),
\]
\[
 F_M^*(s,t)=\max\{F_M(s,t),F_M(t,s)\},
\]
and
\[
 R(s,t)=4\arcsin\bigl(\rho(s)\rho(t)\bigr).
\]
The role in the second column refers to the proposed inequality over all
finite extremal sets, whereas the last column separates theorems from
conjectures based only on the direct database evaluation.

\begin{table}[H]
\centering
\footnotesize
\renewcommand{\arraystretch}{1.16}
\begin{tabular}{>{\raggedright\arraybackslash}p{0.43\textwidth}|
                >{\raggedright\arraybackslash}p{0.14\textwidth}|
                >{\raggedright\arraybackslash}p{0.32\textwidth}}
$F(s,t)$ & Role of $K_4$ & Status \\ \hline
$s+t$ & minimizer & proved; Theorem~\ref{thm:minimal-length}(b) \\
$\sqrt s+\sqrt t$ & minimizer & proved; Theorem~\ref{thm:minimal-length}(b) \\
$s^{-1}+t^{-1}$ & minimizer & proved by $s,t\leq1$ and $m-1\geq3$ \\
$st$ & minimizer & conjectured; database minimum \\
$s^2+t^2$ & minimizer & conjectured; database minimum \\
$\min\{s,t\}$ & minimizer & conjectured; database minimum \\
$\min\{\Theta(s),\Theta(t)\}$ & minimizer & conjectured; database minimum \\
$\sqrt{st}$ & minimizer & conjectured; database minimum \\
$\arcsin s+\arcsin t$ & minimizer & conjectured; database minimum \\
$\tan(s/2)\tan(t/2)$ & maximizer & conjectured; database maximum \\
$\rho(s)\rho(t)$ & maximizer & conjectured; database maximum \\
$F_M^*(s,t)$ & maximizer & Meissner/Blaschke--Lebesgue conjecture; database maximum \\
$R(s,t)$ & maximizer & conjectured; database maximum \\
$\Theta(s)\Theta(t)$ & neither & values occur on both sides of the $K_4$ value
\end{tabular}
\caption{The role of the regular tetrahedron for the objectives currently
implemented in Python or MATLAB.  ``Database minimum/maximum'' means among
the 10,644 supplied embeddings, evaluated directly without optimization.}
\label{tab:direct-objectives}
\end{table}

None of the 9,642 direct evaluations at $m=15,16$ changes the role recorded
in Table~\ref{tab:direct-objectives}.  In particular, every database
minimum/maximum assertion in the table holds on the combined database.
This is an exhaustive statement about the supplied embeddings, not about all
geometrically realizable configurations on the corresponding graphs.

For the first proved row, write $L=2\sin(\theta/2)$ in
Theorem~\ref{thm:minimal-length}(b).  The same theorem applies to
$\sqrt L$ because $\theta\mapsto\sqrt{2\sin(\theta/2)}$ is concave on
$[0,\pi/3]$.  The reciprocal row follows directly: every one of the $m-1$
dual pairs contributes at least two, so the sum is at least six, with equality
for $K_4$.

The spherical product deserves separate emphasis.  Directly on the combined
database, the tetrahedral value and the observed range are, respectively,
\[
 3.2898681,\qquad [3.1988242,3.2993669].
\]
For the 2,195 records at $m=15$ the range is
$[3.2120829,3.2930989]$, and for the 7,447 records at $m=16$ it is
$[3.1988242,3.2975253]$.  The database maximum remains the $m=14$ record
\texttt{n14\_g0483}; no record at $m=15,16$ exceeds it.
Thus the regular tetrahedron itself is neither a minimizer nor a maximizer of this
functional in the database.  The
optimization conjecture below instead concerns a tetrahedral core decorated
by dangling vertices.

The Meissner summand $F_M$ is asymmetric and therefore is not intrinsically
defined on an unordered dual pair.  The orientation-independent row in
Table~\ref{tab:direct-objectives} uses $F_M^*$, which chooses the better of the
two Meissner smoothings pair by pair.

The Python registry is evaluated by \path{EvaluatorAll.py}; the MATLAB-only
rows are evaluated from the same reconstructed chord list using the formulas
in \path{DiameterGraphsFmincon/ObjectiveFunctions.m}.  The CSV and JSON
outputs retain the identifier, objective value, embedding residual, and number
of reconstructed dual pairs for every database entry.

\subsection{Gradient optimization on a fixed graph}

\subsubsection{Optimization method and treatment of collisions}

We next keep a labelled diameter graph $\mathcal D$ and its initial dual-pair
list $\mathcal E$ fixed and optimize the vertex coordinates over the
constraints in \eqref{eq:fixed-feasible-set}.  These are local,
gradient-based nonlinear programs: the Python implementation uses IPOPT, and
the MATLAB implementation uses \texttt{fmincon} with the active-set or
interior-point algorithm.  Analytic first derivatives are supplied for the
objectives and constraints.  The first run for each graph starts from its
database embedding.  Selective restarts use independently prescribed
perturbations of an embedding or of a previous endpoint.  A separate pilot
on thirteen cases also used independent Gaussian point clouds as starting
configurations.  Every completed output is therefore only a possible local
optimum for that fixed graph, not a
proof of its global optimum.

Our main optimized objective is the spherical product
\begin{equation}\label{eq:numerical-product}
 J^{\rm lab}_{\theta,\mathcal E}(X)
 =\sum_{(e,e')\in\mathcal E}\theta(e)\theta(e')
\end{equation}
and we maximize this labelled objective separately on each fixed graph.  For
configurations in $\mathcal A_{\rm geo}(\mathcal D,\mathcal E)$ it agrees with the intrinsic functional
$J_{\rm prod}^{\theta}$ from \eqref{eq:intrinsic-dual-functional}.  For the
1002 graphs with at most 14 vertices, complete first passes were made with
both the MATLAB active-set implementation and, independently, the Python
IPOPT implementation.  The consolidated results below use the MATLAB
computation, which was also used for all graphs at $m=15,16$ and for the
selective restarts.

Collisions require special care.  After each solve we merge connected
components of close vertices, reconstruct all unit-distance pairs from the
coordinates, iteratively remove vertices incident to at most one diameter,
and reconstruct exposed dual arcs on the appropriate intersection circles.
A result with $m_0$ surviving vertices is retained only when it satisfies the
diameter bound, has $2m_0-2$ diameter edges, and has $m_0-1$ exposed dual
pairs.  Finally, the intrinsic value $J_{\rm prod}^{\theta}$ is recomputed from
this reconstructed list.  Thus a solver success flag alone never certifies a
data point, and a value computed with the original labelled pair list is not
used after vertices have merged.

The corrected optimizer uses $10^{-10}$ for IPOPT convergence and $10^{-6}$
for reconstruction.  The same reconstruction tolerance is used in the
MATLAB verification.  The MATLAB results are exported once in a compact
format; the objective and all dual pairs are then recomputed independently in
Python.  The number of displayed decimal places is useful for reproducing the
computations, but does not provide certified error bounds.

\subsubsection{Consolidated spherical-product results}

We call a diameter graph an \emph{iterated tetrahedral decoration} if one can
successively delete vertices of current diameter degree two and obtain
$K_4$.  Likewise, an \emph{iterated pyramidal decoration} reduces to an odd
wheel, namely an odd rim cycle together with a hub adjacent to every rim
vertex.  The recursive definition is necessary: a later decoration can raise
the degree of an earlier dangling vertex from two to three.

The input size $m$ and the number of vertices remaining after reconstruction
need not agree.  We therefore group the results by the size of the supplied
diameter graph.  If the endpoint of the first optimization did not pass the
geometric checks, we repeated the computation from an independently
perturbed embedding.  This was necessary for 6 graphs at $m=13$, 21 at
$m=14$, 59 at $m=15$, and 328 at $m=16$; no such repetition was needed at
the smaller sizes.  Every one of these 414 graphs eventually produced a
verified endpoint.  Since every supplied graph already has a valid database
embedding, a failed first endpoint is a failure of that local optimization
and reconstruction, not an inadmissible graph.

We also repeated the optimization whenever a verified first endpoint did not
contain $K_4$.  There were 34 such cases at $m=15$ and 147 at $m=16$, and
none at the smaller sizes.  For every one of these 181 graphs, a different
initialization produced a verified endpoint containing $K_4$ with a larger
corrected objective than the best non-$K_4$ endpoint found for that graph.
This includes the endpoints which initially had an odd-wheel pyramid.

The objective and constraints were unchanged in these restarts and contained
no term favouring $K_4$; only the decision to restart was triggered by the
absence of $K_4$ from the first verified endpoint.

Table~\ref{tab:product-final-coverage} reports one selected verified endpoint
for every supplied graph. Numerical optimization starts from the embeddings in
the datasets described above for $m\leq16$. Most first endpoints contain
$K_4$. When an endpoint either failed the geometric checks or did not contain
$K_4$, the computation was repeated from a randomized initialization. In every
such case, a verified endpoint containing $K_4$ was obtained. In most cases a
stronger property holds: the diameter graph of the selected endpoint reduces
to $K_4$ after recursive deletion of dangling vertices.
\begin{table}[H]
\centering
\small
\begin{tabular}{c|r|r|r|r|r}
input $m$ & graphs & repeated after first check & contains $K_4$
& reduces to $K_4$ & largest $J_{\rm prod}^{\theta}$\\ \hline
4  & 1    & 0   & 1    & 1    & 3.289868134\\
5  & 0    & 0   & 0    & 0    & --\\
6  & 1    & 0   & 1    & 1    & 3.305074174\\
7  & 1    & 0   & 1    & 1    & 3.320280215\\
8  & 2    & 0   & 2    & 2    & 3.320280212\\
9  & 4    & 0   & 4    & 4    & 3.325130555\\
10 & 11   & 0   & 11   & 11   & 3.337911428\\
11 & 24   & 0   & 24   & 23   & 3.340336592\\
12 & 72   & 0   & 72   & 70   & 3.342761720\\
13 & 212  & 6   & 212  & 198  & 3.343535165\\
14 & 674  & 21  & 674  & 642  & 3.344308574\\
15 & 2195 & 59  & 2195 & 2105 & 3.345238368\\
16 & 7447 & 328 & 7447 & 7016 & 3.345719553\\ \hline
all & 10644 & 414 & 10644 & 10074 & 3.345719553
\end{tabular}
\caption{Selected verified endpoints for the fixed-graph maximization.  The
third column counts first-pass endpoints which had to be repeated after
geometric reconstruction.  The last column is grouped by the input size,
not by the number of surviving vertices.  These are local-search results,
not certified fixed-graph maxima.}
\label{tab:product-final-coverage}
\end{table}

The largest value in the table is obtained from \path{n16_g7145}; its
reconstructed configuration has 14 vertices and is an iterated tetrahedral
decoration.  The largest $m=15$ value comes from \path{n15_g0783}, and the
largest value among inputs through $m=14$ comes from \path{n14_g0001}; these
are also iterated tetrahedral decorations.  In total, 10,074 selected
endpoints reduce to $K_4$, while the remaining 570 contain $K_4$ without
being reducible to it by the degree-two deletion procedure.

The complete run histories, including the successive initializations and
the reconstruction decisions, are kept outside the article in the
machine-readable computation record.  The table can be regenerated by
\path{BuildFinalSphericalProductReport.py}.  The important numerical
conclusion is uniform $K_4$ containment among the selected endpoints; the
computations neither certify the largest local value nor give a global upper
bound for any fixed graph. Examples of numerical results can be seen in Figure \ref{fig:results-sph-product}

\begin{figure}
	\centering
	\includegraphics[width=0.99\textwidth]{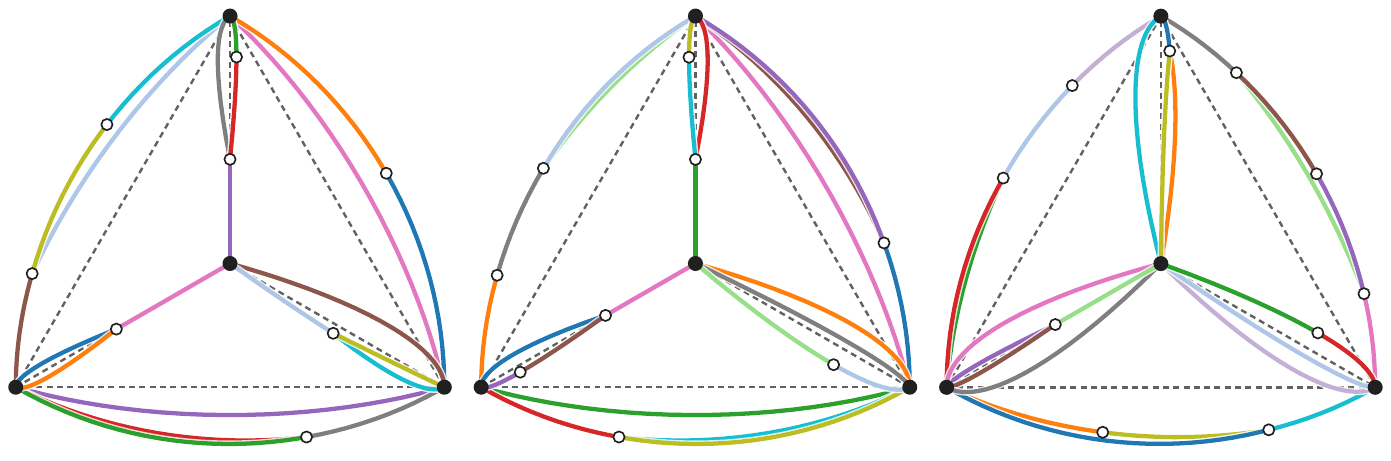}
	\caption{Three examples of numerical optimizers for \eqref{eq:numerical-product}, for diameter graphs with $m=14,15,16$ points. Each displayed configuration contains the regular tetrahedron $K_4$, viewed along a tetrahedral altitude so that one face appears as an equilateral triangle and the fourth vertex projects to its centre. The diameters realizing this regular tetrahedron are drawn with dashed lines. Arcs corresponding to dual edges in the ball polyhedron are drawn with the same color.}
	\label{fig:results-sph-product}
\end{figure}

\subsubsection{Structural conjectures}

The global version suggested by the largest values is the following.

\begin{conj}[$K_4$ containment]\label{conj:product-decoration}
For each $m\geq4$, the supremum of $J_{\rm prod}^{\theta}$ over extremal
sets with at most $m$ vertices equals its supremum over extremal sets whose
diameter graph contains $K_4$.
\end{conj}

This is the statement directly supported by the computations.  For every one
of the 10,644 supplied fixed graphs, the largest corrected value retained from
the starts used for that graph is attained by a verified configuration
containing $K_4$.  In particular, all 181 verified first endpoints without
$K_4$ were improved by a $K_4$-containing restart.  The computations do not
bound the values of unobserved local or global maxima.

A stronger conjecture would require an iterated decoration of $K_4$ rather
than mere containment.  In Table~\ref{tab:product-final-coverage}, 10,074
selected endpoints reduce to $K_4$, whereas 570 contain $K_4$ without such a
reduction.  Thus containment is the uniform numerical conclusion.

The statement for each fixed graph is strictly stronger.

\begin{conj}[$K_4$ containment for a fixed graph]
\label{conj:graphwise-product-decoration}
Let $(\mathcal D,\mathcal E)$ arise from an extremal unit-diameter set.  There
is a maximizing sequence $X^n\in
\mathcal A_{\rm geo}(\mathcal D,\mathcal E)$ for the intrinsic functional
$J_{\rm prod}^{\theta}$ which converges, after rigid motions and passage to a
subsequence, to a finite set $Z$ such that
$\operatorname{core}(Z)$ contains $K_4$ and
\[
 J_{\rm prod}^{\theta}(\operatorname{core}(Z))
 =\sup_{X\in\mathcal A_{\rm geo}(\mathcal D,\mathcal E)}
   J_{\rm prod}^{\theta}(X).
\]
\end{conj}

Together with the exact reformulation \eqref{eq:P1-P2-comparison},
Conjecture~\ref{conj:graphwise-product-decoration} would imply that an
area-minimizing body of constant width one can be chosen to contain a unit
regular tetrahedron.  Indeed, choose extremal sets $X_n$ with
$J_{\rm prod}^{\theta}(X_n)\to\mathfrak B$.  Apply the conjecture to the
fixed graph of each $X_n$ to obtain extremal cores $Y_n$ containing a regular
tetrahedron and satisfying
$J_{\rm prod}^{\theta}(Y_n)\geq J_{\rm prod}^{\theta}(X_n)$.  Choose for
each $Y_n$ a Meissner smoothing $W_n$ as in the proof of
Theorem~\ref{thm:spherical-product-reformulation}.  Then
\[
 2\pi-|\partial W_n|\geq J_{\rm prod}^{\theta}(Y_n)\longrightarrow
 \mathfrak B.
\]
After aligning the tetrahedra by rigid motions, one has
$T\subset W_n\subset B(T)$.  Blaschke selection and continuity of surface
area under Hausdorff convergence give a constant-width limit $W$ with
\[
 T\subset W\subset B(T),\qquad 2\pi-|\partial W|=\mathfrak B.
\]
Thus $W$ minimizes surface area among bodies of constant width one.  Here
$T$ is the regular tetrahedron and $B(T)$ is the corresponding Reuleaux
tetrahedron.  This does not imply that $W$ is a Meissner tetrahedron, since
other constant-width completions of $T$ exist, but it substantially restricts
the class in which an area minimizer must be sought.

Here passage to a limit and reconstruction are essential: the limiting
configuration may have coincident vertices and need not retain the original
labelled graph.
Because $(\mathcal D,\mathcal E)$ is assumed to arise from an extremal set, no
separate ``geometrically realizable'' hypothesis is needed.  The computations
above support but do not prove
Conjecture~\ref{conj:graphwise-product-decoration}.  They compare finitely many
local solutions and provide no certified upper bound for a fixed graph.  They
also show that a proof cannot identify an arbitrary gradient-solver limit
with the desired $K_4$-containing maximizer.

The principal verification and summary programs are:
\begin{itemize}[noitemsep]
\item\path{AnalyzeProductMaximizers.py}
\item\path{AnalyzeFminconResults.py}
\item\path{BuildFinalSphericalProductReport.py}.
\end{itemize}

\section{Conclusions}
\label{sec:conclusions}

Dual-edge functionals provide a common language for geometric inequalities,
mixed surface area, and discrete models of the three-dimensional Blaschke--Lebesgue problem.
The geometric-limit result identifies an extremal core after collisions,
while the linked-partition construction reduces its proposed converse to a
global metric realization question.  Compatible convergence of exposed dual
pairs gives lower semicontinuity of nonnegative intrinsic functionals, while
stable convergence and vanishing on the coordinate axes give continuity; the
general case with coincident vertices remains open.  The additive estimate
in Theorem~\ref{thm:minimal-length} has a regular tetrahedron as a minimizer.  For
products, the spherical functional gives the exact constant-width
area-deficit supremum, whereas its Euclidean counterpart presently gives an
upper bound whose reverse direction is the open centre-hull density question.
The spherical-product functional also leads to a concrete structural
conjecture tested by gradient-based computations on every supplied graph
through 16 vertices.

The combined database supplies exhaustive direct evaluations of 10,644
records and at least one fixed-graph computation starting from every supplied
embedding.  Every endpoint is checked after coincident vertices are merged.
The 414 first endpoints which failed this check were recomputed from different
initializations, and every graph then produced a verified endpoint.  Likewise,
all 181 verified first endpoints without $K_4$ were improved by endpoints
containing $K_4$.  The consolidated report therefore contains a verified
$K_4$-containing endpoint for every supplied graph.  Of these, 10,074 reduce
to $K_4$ by successive degree-two deletions and 570 satisfy containment only.
This supports $K_4$ containment, but the computations are local and the
global maximization problem remains open.

The computations also expose a general methodological point: combinatorial
dual data cannot be transported blindly through a collision.  The
next steps are to prove a reduction principle for iterated dangling
decorations, to prove the proposed $K_4$-containment principle and determine
when it can be strengthened to tetrahedral reduction, to resolve exposure
stability and the centre-hull density question, to settle linked-partition
realization, and to replace local nonlinear optimization by certified upper
bounds on each fixed graph. Proving the conjectures above would significantly
advance the study of the three-dimensional Blaschke--Lebesgue problem.

{\bf Acknowledgments.} The author was partially supported by the French ANR project STOIQUES.

{\bf Competing interests:} The author declares no competing interests.

{\bf Data availability statement:} The data through 14 vertices are from the
repository associated with \cite{meissner_graphs}.  The data for
$m\in\{15,16\}$, together with their generation and validation files, are
available at
\begin{center}\small
\href{https://github.com/beniamin-bogosel/ReuleauxPolyhedraData15-16}%
{\nolinkurl{https://github.com/beniamin-bogosel/ReuleauxPolyhedraData15-16}}.
\end{center}

{\bf Code availability:} The code for reproducing the simulations in Section \ref{sec:numerics} is available at
\begin{center}\small
	\href{https://github.com/beniamin-bogosel/ExtremalDiameterGraphs3D}%
	{\nolinkurl{https://github.com/beniamin-bogosel/ExtremalDiameterGraphs3D}}.
\end{center}

{\bf AI usage statement:} OpenAI Codex was used to orchestrate and automate the numerous optimization tasks in Section \ref{sec:numerics} and to process the results. The author has verified the code and numerical simulations and takes full responsibility for the results presented in this paper.

\bibliographystyle{abbrv}
\bibliography{./biblio}

\end{document}